\documentclass[12pt, reqno]{amsproc}
\usepackage[utf8]{inputenc}
\usepackage[T2A]{fontenc}
\usepackage[english]{babel}
\usepackage{amsmath}
\usepackage{amsfonts}
\usepackage{amssymb}
\usepackage{amsthm}
\usepackage{enumitem}
\usepackage{mathtools}

\usepackage[svgcolors, dvipsnames]{xcolor}
\usepackage{hyperref}
\hypersetup{
    hypertexnames=false,
    colorlinks,
    linkcolor={red!50!black},
    citecolor={blue!50!black},
    urlcolor={blue!80!black}
}

\usepackage{graphicx}
\usepackage[all]{xy}
\usepackage[margin = 2.5cm]{geometry}

\newcommand\mycolor[1]{}

\setlist[enumerate]{itemsep=0.3ex, topsep=0.3ex, label={\rm(\arabic*)}}
\setlist[itemize]{itemsep=0.3ex, topsep=0.3ex, leftmargin=4ex}

\newtheorem{theorem}[subsection]{Theorem}

\newtheorem{subtheorem}[subsubsection]{Theorem}
\newtheorem{sublemma}[subsubsection]{Lemma}

\newtheorem{subcorollary}[subsubsection]{Corollary}

\newtheorem{subremark}[subsubsection]{Remark}
\newtheorem{subexample}[subsubsection]{Example}
\newtheorem{subdefinition}[subsubsection]{Definition}

\makeatletter
\@addtoreset{subsection}{section}
\@addtoreset{equation}{section}
\@addtoreset{figure}{section}
\@addtoreset{table}{section}
\makeatother

\makeatletter
\newcommand\testshape{family=\f@family; series=\f@series; shape=\f@shape.}
\def\myemphInternal#1{\if n\f@shape%
\begingroup\itshape #1\endgroup\/%
\else\begingroup\sf\itshape\small #1\endgroup%
\fi}
\def\myemph{\futurelet\testchar\MaybeOptArgmyemph}
\def\MaybeOptArgmyemph{\ifx[\testchar \let\next\OptArgmyemph
                 \else \let\next\NoOptArgmyemph \fi \next}
\def\OptArgmyemph[#1]#2{\index{#1}\myemphInternal{#2}}
\def\NoOptArgmyemph#1{\myemphInternal{#1}}
\makeatother

\newcommand\term[2][\empty]{\myemph[#1]{#2}}

\newcommand\Aman{A}
\newcommand\Bman{B}
\newcommand\Cman{C}

\newcommand\Kman{K}
\newcommand\Lman{L}
\newcommand\Mman{M}
\newcommand\Nman{N}

\newcommand\Pman{P}

\newcommand\Uman{U}

\newcommand\Wman{W}

\newcommand\Yman{Y}

\newcommand\bN{\mathbb{N}}

\newcommand\bR{\mathbb{R}}
\newcommand\bZ{\mathbb{Z}}

\newcommand\id{\mathrm{id}}          
\newcommand\Int{\mathrm{Int}}        
\newcommand\supp{\mathrm{supp\,}}    
\newcommand\eps{\varepsilon}                   

\newcommand\restr[2]{#1\vert_{#2}}
\newcommand\GL{\mathrm{GL}}

\newcommand\Diff{\mathcal{D}}       
\newcommand\Homeo{\mathcal{H}}      
\newcommand\Iso{\mathrm{Iso}}       
 
\newcommand\Cr[1]{\mathcal{C}^{#1}}

\newcommand\VBAut[2][\empty]{\GL(#2\ifx\empty #1\relax\else,#1\fi)}

\newcommand\DiffLP{\Diff}  
\newcommand\DiffLPInv[3][\empty]{\DiffLP_{inv}(#2,#3\ifx\empty#1\relax\else,#1\fi)}

\newcommand\DiffLPFix[3][\empty]{\DiffLP_{fix}(#2,#3\ifx\empty#1\relax\else,#1\fi)}

\newcommand\DiffLPNb[3][\empty]{\DiffLP_{nb}(#2,#3\ifx\empty#1\relax\else,#1\fi)}

\newcommand\adif{a}
\newcommand\bdif{b}
\newcommand\cdif{c}
\newcommand\ddif{d}

\newcommand\hadif{\hat{\adif}}
\newcommand\hbdif{\hat{\bdif}}
\newcommand\hcdif{\hat{\cdif}}

\newcommand\dif{h}
\newcommand\gdif{g}
\newcommand\kdif{k}

\newcommand\px{x}

\newcommand\acolr{black}
\newcommand\bcolr{black}
\newcommand\ccolr{black}

\newcommand\ucolr{black}
\newcommand\vcolr{black}
\newcommand\wcolr{black}

\newcommand\putcolor[2][\empty]{\ifx\empty#1#2\else{\color{#1}{#2}}\fi}
\newcommand\lowerindex[1]{\ifx\empty#1\relax\else_{#1}\fi}
\newcommand\upperindex[1]{\ifx\empty#1\relax\else^{#1}\fi}

\newcommand\manifNotation[2][\empty]{\putcolor[#1]{#2}}
\newcommand\chartNotation[2][\empty]{\putcolor[#1]{\mathsf{#2}}}
\newcommand\atlasNotation[2][\empty]{\putcolor[#1]{\mathsf{#2}}}

\newcommand\indexNotation[2][\empty]{\putcolor[#1]{#2}}
\newcommand\indexSetNotation[2][\empty]{\putcolor[#1]{#2}}
\newcommand\categoryNotation[2][\empty]{\putcolor[#1]{\mathcal{#2}}}

\newcommand\indatl[2]{#1_{*}(#2)}

\newcommand\mpair[2]{#1^{(#2)}}

\newcommand\Cone[1][\empty]{\Kman\lowerindex{#1}}

\newcommand\Aspace{\manifNotation[\acolr]\Mman}
\newcommand\Bspace{\manifNotation[\bcolr]\Nman}

\newcommand\Aatlas{\atlasNotation[\acolr]{A}}
\newcommand\Batlas{\atlasNotation[\bcolr]{B}}

\newcommand\Latlas{\atlasNotation[\ccolr]{M}}
\newcommand\CanonAtlas[1]{\atlasNotation[\ccolr]{C}_{#1}}

\newcommand\Uatlas{\atlasNotation[\ucolr]{U}}
\newcommand\Vatlas{\atlasNotation[\vcolr]{V}}
\newcommand\Watlas{\atlasNotation[\wcolr]{W}}

\newcommand\AutCObject[2][\empty]{\mathrm{Aut}\lowerindex{\categoryNotation[\acolr]{#1}}(#2)}

\newcommand\Uchart[1][\empty]{\chartNotation[\acolr]{u}\lowerindex{#1}}
\newcommand\Vchart[1][\empty]{\chartNotation[\bcolr]{v}\lowerindex{#1}}

\newcommand\Udomain[1][\empty]{\manifNotation[\acolr]{U}\lowerindex{#1}}
\newcommand\Vdomain[1][\empty]{\manifNotation[\bcolr]{V}\lowerindex{#1}}
\newcommand\Wdomain[1][\empty]{\manifNotation[\ccolr]{W}\lowerindex{#1}}

\newcommand\hUdomain[1][\empty]{\manifNotation[\acolr]{\tilde{U}}\lowerindex{#1}}
\newcommand\hVdomain[1][\empty]{\manifNotation[\bcolr]{\tilde{V}}\lowerindex{#1}}
\newcommand\hWdomain[1][\empty]{\manifNotation[\ccolr]{\tilde{W}}\lowerindex{#1}}

\newcommand\Ucmap[1][\empty]{\chartNotation[\acolr]{\phi}\lowerindex{#1}}
\newcommand\Vcmap[1][\empty]{\chartNotation[\bcolr]{\psi}\lowerindex{#1}}

\newcommand\hUcmap[1][\empty]{\chartNotation[\acolr]{\phi_{\#}}\lowerindex{#1}}
\newcommand\hVcmap[1][\empty]{\chartNotation[\bcolr]{\psi_{\#}}\lowerindex{#1}}

\newcommand\trmap[1]{\gdif_{#1}}
\newcommand\Atrmap{\trmap{\Aatlas}}
\newcommand\Btrmap{\trmap{\Batlas}}

\newcommand\uind{\indexNotation[\acolr]{i}}
\newcommand\vind{\indexNotation[\bcolr]{j}}

\newcommand\Uindset{\indexSetNotation[\acolr]{I}}

\newcommand\canUcmap[1][\empty]{\chartNotation[\acolr]{\Phi}\lowerindex{#1}}
\newcommand\canVcmap[1][\empty]{\chartNotation[\bcolr]{\Psi}\lowerindex{#1}}

\newcommand\wuinc{\putcolor[\ucolr]{i}}
\newcommand\wvinc{\putcolor[\vcolr]{j}}

\newcommand\hwuinc{\putcolor[\ucolr]{\tilde{i}}}
\newcommand\hwvinc{\putcolor[\vcolr]{\tilde{j}}}

\newcommand\AtCat[1][\empty]{\mathsf{At}\ifx\empty#1\relax\else^{#1}\fi}

\newcommand\uva[1]{\atlasNotation{Q}_{#1}}

\newcommand\XDIFF[3]{\Diff^{#1}_{#2}(#3)}

\newcommand\erestr[2]{#1_{\lfloor\!#2}}

\newcommand\HomeoLUV{\Homeo(\Aspace,\Udomain,\Vdomain)}
\newcommand\HomeoLPartUV{\Homeo(\Aspace,\{\Udomain,\Vdomain\})}

\newcommand\DiffUW{\XDIFF{}{}{\Udomain,\Wdomain}}
\newcommand\DiffVW{\XDIFF{}{}{\Vdomain,\Wdomain}}
\newcommand\DiffW{\XDIFF{}{}{\Wdomain}}
\newcommand\EDiffWU{\mathcal{E}_{\Udomain}(\Wdomain)}
\newcommand\EDiffWV{\mathcal{E}_{\Vdomain}(\Wdomain)}

\newcommand\DiffUVW{\XDIFF{}{}{(\Udomain,\Wdomain),(\Vdomain,\Wdomain)}}
\newcommand\DiffVUW{\XDIFF{}{}{(\Vdomain,\Wdomain),(\Udomain,\Wdomain)}}

\newcommand\aUcmap{\Ucmap}
\newcommand\aVcmap{\Vcmap}
\newcommand\haUcmap{\hUcmap}
\newcommand\haVcmap{\hVcmap}

\newcommand\xbar[1]{\tilde{#1}}

\newcommand\bUcmap{\xbar{\Ucmap}}
\newcommand\bVcmap{\xbar{\Vcmap}}
\newcommand\hbUcmap{\bUcmap_{\#}}
\newcommand\hbVcmap{\bVcmap_{\#}}

\newcommand\UVAtlas{(\Udomain,\Vdomain)}

\newcommand\Astruct{\mathfrak{A}}
\newcommand\Bstruct{\mathfrak{B}}

\newcommand\Diffr[3][\empty]{\Diff_{#1}^{#2}(#3)}

\newcommand\AtlasCompletion[1]{\overline{#1}}
\newcommand\AtlasRestr[2]{\restr{#1}{#2}}

\newcommand\kk{r}
\newcommand\Ck{\Cr{\kk}}

\newcommand\XCMAP[3]{\mathcal{C}^{#1}_{#2}(#3)}

\newcommand\DLine{\mathbb{L}}    
\newcommand\Ylet{\mathbb{Y}}     

\newcommand\Rzp{\bR\setminus0}
\newcommand\Rpos{\mathbb{R}_{pos}}

\newcommand\HRplz{\Homeo^{+}_{0}(\bR)}
\newcommand\DRplz{\XDIFF{+}{0}{\bR}}
\newcommand\DRpos{\XDIFF{+}{}{\Rpos}}
\newcommand\DExtRneg{\mathcal{E}\!xt(\Rpos)}
\newcommand\calE{\mathcal{E}}

\newcommand\dbl[3]{#1\setminus #2 \,/\, #3}

\newcommand\dbli[2]{#1\setminus #2^{\pm1} \,/\, #1}

\newcommand\JDiffRz[1][\kk]{\mathcal{E}^{#1}(\bR,0)}
\newcommand\JDiffRzOrPres[1][\kk]{\JDiffRz[\kk,+]}
\newcommand\JDiffRzOrRev[1][\kk]{\JDiffRz[\kk,-]}

\newcommand\AutWExtU{\mathcal{E}_{\Ccateg}(\Wman,\wuinc)}
\newcommand\AutWExtV{\mathcal{E}_{\Ccateg}(\Wman,\wvinc)}

\newcommand\udif{\alpha}
\newcommand\vdif{\beta}
\newcommand\wdif{\gamma}
\newcommand\uvdif{\eta}

\newcommand\RTwoCopies{\bR\times\{0,1\}}

\newcommand\Ccateg{\categoryNotation[\acolr]{C}}
\newcommand\Hcateg{\categoryNotation[\bcolr]{H}}

\newcommand\AutC[1]{\AutCObject[\Ccateg]{#1}}
\newcommand\AutH[1]{\AutCObject[\Hcateg]{#1}}

\newcommand\AutCU{\AutC{\Udomain}}

\newcommand\AutCW{\AutC{\Wdomain}}
\newcommand\AutHU{\AutH{\Udomain}}

\newcommand\IsomCat[3]{\Iso_{#1}(#2,#3)}
\newcommand\IsomH[2]{\IsomCat{\Hcateg}{#1}{#2}}
\newcommand\IsomC[2]{\IsomCat{\Ccateg}{#1}{#2}}

\newcommand\Cspan{\Ccateg}
\newcommand\HCatlas{\Hcateg\Ccateg}
\newcommand\spanCat[5]{#1 \xleftarrow{~#2~} #3 \xrightarrow{~#4~} #5}
\newcommand\spanRev[1]{#1^{*}}
\newcommand\MorCat[3]{\mathcal{M}_{#1}(#2,#3)}

\newcommand\MorH[2]{\MorCat{\Hcateg}{#1}{#2}}

\newcommand\ArrCat[1]{#1^{\mathbf{2}}}
\newcommand\ArrCcateg{\ArrCat{\Ccateg}}
\newcommand\ArrHcateg{\ArrCat{\Hcateg}}

\title{Differentiable structures on a union of two open sets}

\author{Mykola Lysynskyi}
\address{Department of Algebra and Topology, Institute of Mathematics of NAS of Ukraine, Tereshchenkivska str. 3, Kyiv, 01601, Ukraine}
\email{m.lysynskyi@imath.kiev.ua}

\author{Sergiy Maksymenko}
\address{Department of Algebra and Topology, Institute of Mathematics of NAS of Ukraine, Tereshchenkivska str. 3, Kyiv, 01601, Ukraine}
\email{maks@imath.kiev.ua}


\begin{document}

\begin{abstract}
In a recent paper the authors classified differentiable structures on the non-Hausdorff one-dimensional manifold $\mathbb{L}$ called the \emph{line with two origins} which is obtained by gluing two copies of the real line $\mathbb{R}$ via the identity homeomorphism of $\mathbb{R}\setminus 0$.

Here we give a classification of differentiable structures on another non-Hausdorff one-dimensional manifold $\mathbb{Y}$ (called \emph{letter  ``$Y$}'') obtained by gluing two copies of $\mathbb{R}$ via the identity map of positive reals.
It turns out that, in contrast to the real line, for every $r=1,\ldots,\infty$, both manifolds $\DLine$ and $\mathbb{Y}$ admit uncountably many pair-wise non-diffeomorphic $\mathcal{C}^{k}$-structures.

We also observe that the proofs of these classifications are very similar.
This allows to formalize the arguments and extend them to a certain general statement about arrows in arbitrary categories.
\end{abstract}

\maketitle


\section{Introduction}
Classification of smooth structures on manifolds is an extremely hard problem.
It is completely solved for Hausdorff manifolds of dimensions $m\leq 3$: namely any two smooth structures on such a manifold are diffeomorphic, \cite{Munkres:PhD:1955} for $n=2$, \cite[Theorem~6.3]{Munkres:AnnMath:1968} for $m=3$, and \cite[Corollary~1.18]{Whitehead:AnnMath:1961} for $n\leq3$.
In higher dimensions there is a lot of particular statements, see e.g.\ \cite{Stallings:PCPS:1962, Freedman:JDG:1982, Donaldson:JDG:1983, Gompf:JDG:1985, Taubes:JDG:1987, FreedUhlenbeck:Inst:1991, ChernovNemirovski:CMPh:2013, Akbulut:PM:2012} and references therein.
It is also worth to mention the result by Kirby and Siebenmann~\cite[Theorem on page 155]{KirbySiebenmann:Smooth:1977} implying that on every closed topological Hausdorff manifold of dimension $\geq5$ there exists only finitely many classes of non-diffeomorphic smooth structures.

On the other hand, for non-Hausdorff manifolds the situation is different: even simplest manifolds of dimension $1$ can admit infinitely many pairwise non-diffeomorphic smooth structures, \cite{Nyikos:AM:1992}.
Also, in a recent paper by the authors~\cite{LysynskyiMaksymenko:SmoothStr:2024} it is obtained a complete classification of $\Cr{\kk}$ structures on the non-Hausdorff line $\DLine$ with two origins, obtained by gluing two copies of $\bR$ by the identity diffeomorphism of $\bR\setminus\{0\}$, see Figure~\ref{fig:LY} and Theorem~\ref{th:Ck_struct_on_L_summary} below.
Let us briefly recall that statement.

\begin{figure}[htbp!]
\centering\includegraphics[height=3.5cm]{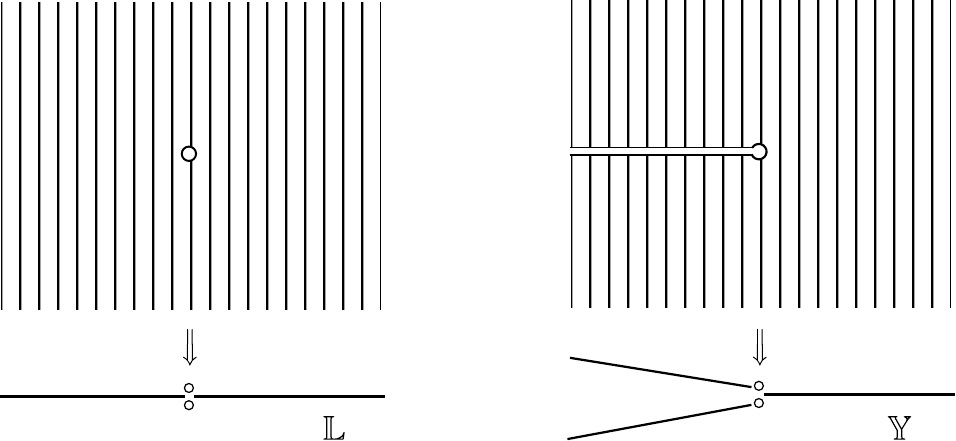}
\caption{Non-Hausdorff manifolds $\DLine$ and $\Ylet$ as leaf spaces of foliations}
\label{fig:LY}
\end{figure}

Let $\bZ_{2}=\{\pm1\}$ be the cyclic group of order $2$.
For a group $W$ let $W\wr\bZ_{2}$ be the \term{wreath product} of $W$ and $\bZ_{2}$, which can be defined as a Cartesian product $W \times W \times \bZ_{2}$ with the following operation:
\begin{align*}
   (c,d,\delta)(a,b,+1) &:= (ca,bd,\delta), &
   (c,d,\delta)(a,b,-1) &:= (da,cb,-\delta),
\end{align*}
for $a,b,c,d\in W$ and $\delta \in\bZ_2$.
Now, if $W$ is a subgroup of some other group $D$, then there is a natural left action of $W\wr\bZ_{2}$ on $D$ given by
\[
    (a,b,\delta)\cdot g := (b g a^{-1})^{\delta}.
\]
The corresponding quotient set will be denoted by $\dbli{W}{D}$.
Evidently, the orbit of an element $g\in D$ coincides with the set $W g W \cup W g^{-1} W$.
It will be called the \term{$(W,\pm)$-double coset} of $g$, or even $\pm$-double cosets, if the subgroup $W$ is known from the context.

Note also that $W\times W \times 1$ is a subgroup of $W\wr\bZ_{2}$ and thus it also acts on $D$.
The quotient set of that action is usually denoted by $\dbl{W}{D}{W}$ and called \term{$W$-double cosets}.
Evidently, in this case the orbit of $g\in D$ is $W g W$.

\begin{theorem}[\cite{LysynskyiMaksymenko:SmoothStr:2024}]
\label{th:Ck_struct_on_L_summary}
Let $\kk\in\{1,\ldots,\infty\}$, $\mathcal{D}:=\XDIFF{+,\kk}{}{\bR\setminus0}$ be the group of orientation-preserving $\Ck$-diffeomorphisms of $\bR\setminus0$, and $\mathcal{W}$ be its subgroup consisting of diffeomorphisms which can be extended to a $\Ck$-diffeomorphism of $\bR$.
Then there is a canonical bijection between
\begin{itemize}
\item
isomorphism classes of $\Cr{\kk}$ structures on $\DLine$ and
\item
$(\mathcal{W},\pm)$-double coset classes $\dbli{\mathcal{W}}{\mathcal{D}}$.
\end{itemize}
In particular, it follows that $\DLine$ admits uncountably many pair-wise non-diffeomorphic $\Ck$-structures.
\end{theorem}

The aim of the present paper is to show that the technique of~\cite{LysynskyiMaksymenko:SmoothStr:2024} can be seen as a certain formal and rather general statement about arrows in arbitrary categories.
In particular, using similar arguments we will give a classification of smooth structures on another example of a non-Hausdorff manifold, the \term{letter $\Ylet$}, obtained by gluing two copies of $\bR$ by the identity homeomorphism of the set $\Rpos:=(0;+\infty)$ of positive reals, see Figure~\ref{fig:LY}.
Existence of smooth structures on $\Ylet$ is shown in~\cite{HaefligerReeb:EM:1957}.
The following statement classifies such structures and summarizes results obtained in Section~\ref{sect:diff_struct_on_Y}, see e.g.\ Theorem~\ref{th:Ck_str_on_Y}:
\begin{theorem}\label{th:Ck_struct_on_Y_summary}
Let $\kk\in\{1,\ldots,\infty\}$, $\mathcal{D}:=\DRpos$ be the group of orientation-preserving $\Ck$-diffeomorphisms of $\Rpos$, and $\mathcal{W}$ be its subgroup consisting of diffeomorphisms which can be extended to a $\Ck$-diffeomorphism of $\bR$.
There is a canonical bijection between
\begin{itemize}
\item
isomorphism classes of $\Cr{\kk}$ structures on $\Ylet$ and
\item
$(\mathcal{W},\pm)$-double coset classes $\dbli{\mathcal{W}}{\mathcal{D}}$.
\end{itemize}
In particular, it follows that $\Ylet$ admits uncountably many pair-wise non-diffeomorphic $\Ck$-structures.
\end{theorem}

\begin{subremark}\rm
The proof of Theorem~\ref{th:Ck_struct_on_Y_summary} will be given in Section~\ref{sect:diff_struct_on_L}, see Theorem~\ref{th:Ck_str_on_Y} and Corollary~\ref{cor:nondif_ck_struct_on_Y}.
It is similar to the proof of Theorem~\ref{th:Ck_struct_on_L_summary} and therefore for the convenience of the reader let us briefly review their common arguments.

It is well known and is easy to see that for every connected Hausdorff manifold $\Aspace$, its group of homeomorphisms $\Homeo(\Aspace)$ transitively acts on the interior $\Int{\Aspace}$, so $\Int{\Aspace}$ is a unique proper subset invariant under $\Homeo(\Aspace)$.
On the other hand, if $\Aspace$ is not Hausdorff, then the set of points (called \term{branch points}) $\Sigma_{\Aspace}$ in which $\Aspace$ ``looses'' Hausdorff property is also invariant under $\Homeo(\Aspace)$.

For instance, let $\Aspace$ be either $\DLine$ or $\Ylet$.
Then $\Aspace$ is glued with two copies of $\bR$ which can be regarded then as open subsets $\Udomain$ and $\Vdomain$, see Figures~\ref{fig:y_letter} and~\ref{fig:l_line} below.
Moreover, $\Aspace$ has exactly two branch points and $\Udomain$ and $\Vdomain$ are open neighborhoods of those points.
One easily checks that every homeomorphism of $\Aspace$ leaves those subsets invariant or interchange them.

Note that then any particular pair of homeomorphisms $\Ucmap\colon\Udomain\to\bR$ and $\Vcmap\colon\Vdomain\to\bR$ such that the corresponding ``transition map'' $\Vcmap\circ\Ucmap\colon\Ucmap(\Udomain\cap\Vdomain)\to\Vcmap(\Udomain\cap\Vdomain)$ is a $\Ck$-diffeomorphism, can be regarded as a $\Ck$-atlas on $\Aspace$.
Now given two such atlases $\Aatlas = (\aUcmap,\aVcmap)$ and $\Batlas = (\bUcmap,\bVcmap)$ on $\Aspace$, every homeomorphism $\dif\colon\Aspace\to\Bspace$ can be written down in terms of the corresponding charts.
This gives a certain (equivalence) relation between the transition maps $\bVcmap\circ\bUcmap^{-1}$ and $\aVcmap\circ\aUcmap^{-1}$ and coordinate representations of $\dif$ in those charts.
In particular, $\dif$ is a diffeomorphism iff those transition maps belong to the same $\pm$-double coset $\dbli{\mathcal{W}}{\mathcal{D}}$, where $\mathcal{D}$ and $\mathcal{W}$ are the respective groups from Theorems~\ref{th:Ck_struct_on_L_summary} and~\ref{th:Ck_struct_on_Y_summary}.
These theorems thus conclude that $\dbli{\mathcal{W}}{\mathcal{D}}$ classifies $\Ck$-structures on $\Aspace$.
\end{subremark}

The previous observation suggests considering a more general situation when a manifold $\Aspace$ is glued from two differentiable manifolds $\Udomain$ and $\Vdomain$ by some diffeomorphism between their open subsets.
Then one can classify in a similar way differentiable structures on $\Aspace$ up to a diffeomorphisms leaving $\Udomain$ and $\Vdomain$ invariant or interchanging them (if this is possible).
This will be done in Theorems~\ref{th:CkLUV_HLUV} and~\ref{th:act_DUWZ2_DW}.
In particular, in Section~\ref{sect:diff_struct_on_L}, we deduce from them Theorem~\ref{th:Ck_struct_on_L_summary}.
Moreover, we also indicate that the developed technique was implicitly presented in the proof of uniqueness of differentiable structures on the real line and the circle, see Theorem~\ref{th:uniq_ck_struct_on_R} and Lemma~\ref{lm:diff_to_canon_struct}.

\begin{subremark}\rm
The usual scheme of classifying differentiable structures is based on the result of J.H.C.~Whitehead, \cite[Theorem~7]{Whitehead:AnnM:1940}, that every differentiable manifold admits a triangulation being ``smooth'' over each closed simplex, and different variants of approximations theorems claiming that every homeomorphism of a triangulated compact manifold can be uniformly approximated by PL-homeomorphism with respect to some subdivision, e.g.\ \cite{Cairns:AnnM:1936, Moise_V:AnnM:1952, Connell:AnnM:1963}.
Then the problem of whether two homeomorphic differentiable manifolds are diffeomorphic reduces to ``smoothing'' PL-homeomorphisms.

In Section~\ref{sect:atlases_with_two_charts} we thus consider an open cover of a manifold $\Aspace$ consisting of two subsets $\Udomain$ and $\Vdomain$, so the \term{nerve} of this cover is just a $1$-dimensional simplex $[0;1]$.
In spite that we ``mimic'' such a very simple triangulation, this allows to discover a formalism of $\pm$-double cosets which was not presented in the known results on the classification of differentiable structures.
\end{subremark}

In the last Section~\ref{sect:categorical_view} we further extend the previous results to the context of categories.
Assume that we have a category $\Hcateg$ and its subcategory $\Ccateg$.
(Think of them as of the category of homeomorphisms of manifolds and its subcategory consisting of $\Ck$-diffeomorphisms).
Given a diagram $\Aspace\colon\spanCat{\Udomain}{\wuinc}{\Wdomain}{\wvinc}{\Vdomain}$ in $\Ccateg$, called \term{$\Ccateg$-span}, we define a notion of an \term{$\HCatlas$-atlas} on $\Aspace$, see Definition~\ref{def:HC_atlas}, and classify such atlases up to an isomorphism (Theorems~\ref{th:CkLUV_HLUV:categories} and~\ref{lm:act_DUWZ2_DW_categories}).
As an illustration we will give a characterization of double cosets in terms of those atlases, see Corollary~\ref{cor:char_double_cosets}.

\begin{subremark}\rm
One-dimensional, especially non-Hausdorff, manifolds often appear as spaces of leaves of foliations on surfaces, see e.g.~\cite{HaefligerReeb:EM:1957, Gauld:NonMetrManif:2014} and Figure~\ref{fig:LY}.
For instance, let $\mathcal{F} =\{x = c \mid c\in\bR \}$ be the foliation on $\bR^2$ by vertical lines.
Given an open subset $\Udomain\subset\bR^2$, let $\mathcal{F}_{\Udomain}$ be the induced foliation on $\Udomain$ whose elements are \term{connected components} of the intersections of $\Udomain$ with the leaves of $\mathcal{F}$.
Let also $\Udomain'$ be the corresponding space of leaves, and $p\colon\Udomain\to\Udomain'$ be the corresponding quotient map.
Endow $\Udomain'$ with the respective quotient topology, so a subset $A\subset\Udomain'$ is open iff $p^{-1}(A)$ is open in $\Udomain$.
One easily checks that $\Udomain'$ is a one-dimensional manifold, however it may loose Hausdorff property.

For example, if $\Lman = \bR^2 \setminus (0,0)$ and $\Yman= \bR^2 \setminus \bigl( (-\infty;0]\times 0 \bigr)$, then the above manifolds $\DLine$ and $\Ylet$ are the respective spaces of leaves of the induced foliations $\mathcal{F}_{\Lman}$ and $\mathcal{F}_{\Yman}$.
\end{subremark}

\section{Differentiable structures on manifolds}\label{sect:diff_structs_on_manif}
In this section we will briefly recall standard definitions related with differentiable structures on manifolds.
The reader can skip this section and use it for references.

\subsection{Restriction map}
\label{sect:restriction_map}
Let $\Udomain$ be a set, $\Wdomain$ be a subset of $\Udomain$, and $\wuinc\colon\Wdomain\subset\Udomain$ be the inclusion map.
If $\dif\colon\Udomain\to\Vdomain$ is a map into another set, then the composition $\dif\circ\wuinc\colon\Wdomain\to\Vdomain$ is called the \term{restriction to $\Wdomain$ map} and denoted by $\restr{\dif}{\Wdomain}$.

In what follows, we will consider maps $\dif\colon\Udomain\to\Udomain$ such that $\dif(\Wdomain)\subset\Wdomain$.
Note that $\dif$ yields then a unique map $\erestr{\dif}{\Wdomain}\colon\Wdomain\to\Wdomain$ given by $\erestr{\dif}{\Wdomain}(\px)=\dif(\px)$ for $\px\in\Wdomain$.
Though it is natural to still regard $\erestr{\dif}{\Wdomain}$ as a restriction of $\dif$ to $\Wdomain$, it is not the same as $\restr{\dif}{\Wdomain}$, since formally they have distinct target spaces.
In fact, they are related by the identity $\restr{\dif}{\Wdomain}:=\dif\circ\wuinc = \wuinc \circ \erestr{\dif}{\Wdomain}$, which means commutativity of the following diagram:
\[
    \xymatrix@C=6em{
        \Wdomain \ar[d]_-{\erestr{\dif}{\Wdomain}} \ar@{^(->}[r]^-{\wuinc} \ar[rd]^-{\restr{\dif}{\Wdomain}}&
        \Udomain \ar[d]^-{\dif} \\
        \Wdomain \ar@{^(->}[r]^-{\wuinc} &
        \Udomain
    }
\]
Of course, the difference between $\restr{\dif}{\Wdomain}$ and $\erestr{\dif}{\Wdomain}$ is rather trivial and is usually neglected, however it is visible from the above diagram and will become essential in Section~\ref{sect:categorical_view} when we extend our results to general categories.

\subsection{Manifolds}
\label{sect:manifolds}
A \term{cone} in a finite-dimensional linear space $\Pman$ is a subset of the form
\[\Cone := \{ \px\in\Pman \mid f_i(\px)\geq0, i=1,\ldots,k\}\]
where $f_1,\ldots,f_k\colon\Pman\to\bR$ is any finite collection of linear functions and such that its \term{interior},
\[
    \Int{\Cone} := \{ \px\in\Pman \mid f_i(\px)>0, i=1,\ldots,k\},
\]
is non-empty.
Then $\partial\Cone := \Cone\setminus\Int{\Cone}$ is called the \term{boundary} of $\Cone$.

Let $\Aspace$ be a topological space.
By a \term{chart} $\Uchart=(\Udomain,\Ucmap)$ on $\Aspace$ we will mean an open embedding (i.e.\ a homeomorphism onto open subset) $\Ucmap\colon\Udomain\to\Cone$ into some cone of some Euclidean space.
For every such pair the set $\Udomain$ is called the \term{domain}, while $\Ucmap$ is the \term{coordinate map} of $\Uchart$.
It will be formally convenient to allow charts with empty domain $\Udomain$.

A \term{partial atlas} on $\Aspace$ is an arbitrary collection $\Aatlas = \{(\Udomain[\uind],\Uchart[\uind])\}_{\uind\in\Uindset}$ of charts.
In this case the union $\supp\Aatlas = \cup_{\uind\in\Uindset}\Udomain[\uind]$ of their domains will be called the \term{support} of $\Aatlas$.
A partial atlas $\Aatlas$ on $\Aspace$ is called \term{atlas} whenever $\supp\Aatlas=\Aspace$, i.e.\ when the domains of charts of $\Aatlas$ cover $\Aspace$.

A topological space $\Aspace$ will be called%
\footnote{Spaces admitting atlases are usually called \term{locally Euclidean}, while by a \term{manifold} one usually means a locally Euclidean space being also Hausdorff and having countable base.
For brevity we will use the term \term{manifold} and do not assume in general neither Hausdorff property nor existence of countable base.}
a \term{manifold} if every point $\px\in\Aspace$ has an open neighborhood homeomorphic to some open subset of some cone in some Euclidean space.
This is equivalent to the assumption that $\Aspace$ admits an atlas.

It is easy to see that every manifold is $T_1$ (actually every locally $T_1$-space is $T_1$ itself).

Let $\Aspace$ be a manifold.
A point $\px\in\Aspace$ is an \term{interior} point, if there exists a chart $\Ucmap\colon\Udomain\to\Cman$ into some cone such that $\Ucmap(\px) \in \Int{\Cman}$.
The set of all interior points of $\Aspace$ is denoted by $\Int{\Aspace}$ and called the \term{interior} of $\Aspace$, while its complement $\partial\Aspace = \Aspace\setminus\Int{\Aspace}$ is the \term{boundary} of $\Aspace$.

\subsection{Differentiable structures}
\label{sect:diff_structs_on_manifs}
For $i=1,2$ let $\Cone[i]\subset\bR^{m_i}$ be a cone, $\Uman_i\subset\Cone[i]$ an open subset, and  $\dif\colon\Uman_1\to\Uman_2$ be a map.
Then $\dif$ is (\term{differentiable of class}) $\Ck$ if it extends to a $\Ck$ map $\dif'\colon\Uman'_1\to\Uman'_2$ defined on some open subsets $\Uman'_i\subset\bR^{n_i}$ with $\Uman_i\subset\Uman'_i$, $i=1,2$.
In that case $\dif$ is a \term{$\Ck$-diffeomorphism} if it is a homeomorphism and its inverse $\dif^{-1}$ is $\Ck$ as well.

Let $\Uchart=(\Ucmap\colon\Udomain\to\Cone')$ and $\Vchart=(\Vcmap\colon\Vdomain\to\Cone)$ be two \term{overlapping} charts on a manifold $\Aspace$, i.e.\ $\Udomain\cap\Vdomain \ne \varnothing$.
Then the following well-defined homeomorphism
\[
   \Cone' \supset \Ucmap(\Udomain\cap\Vdomain)\xrightarrow{~\Vcmap\circ\Ucmap^{-1}~}
                  \Vcmap(\Udomain\cap\Vdomain) \subset \Cone
\]
between open subsets of some cones is called the \term{transition map} from $\Uchart$ to $\Vchart$.
One can also regard $\Vcmap\circ\Ucmap^{-1}$ as an \term{open embedding} $\Vcmap\circ\Ucmap^{-1}\colon\Ucmap(\Udomain\cap\Vdomain) \to \Cone$.
The latter point of view will be very useful.

Let also $\kk\in\bN\cup\{\infty\}$.
Then $\Uchart$ and $\Vchart$ are called \term{$\Ck$-compatible} if either $\Udomain\cap\Vdomain=\varnothing$ or the transition map $\Vcmap\circ\Ucmap^{-1}$ is a $\Ck$-embedding.
The latter requirement can be rephrased by saying that $\Vcmap\circ\Ucmap^{-1}\colon\Ucmap(\Udomain\cap\Vdomain)\to\Cone$ is an \term{open $\Ck$-embedding}.

An (partial) atlas on $\Aspace$ is a \term{$\Ck$-atlas} if its charts are pair-wise $\Ck$-compatible.

\begin{subexample}\rm
1) A single chart $(\Udomain,\Ucmap)$ is a partial $\Ck$-atlas with support $\Udomain$ for all $\kk=1,\ldots,\infty$.

2) Let $\Aatlas=\{(\Udomain[\uind],\Ucmap[\uind])\}_{\uind\in\Uindset}$ be a partial $\Ck$-atlas on $\Aspace$.
Then for every open subset $\Vdomain\subset\Aspace$ the following collection of charts
\begin{equation}\label{equ:atlas_restr}
    \AtlasRestr{\Aatlas}{\Vdomain} :=
    \bigl\{
        (\Vdomain\cap\Udomain[\uind], \restr{\Ucmap[\uind]}{\Vdomain\cap\Udomain[\uind]})
        \mid
        \Vdomain\cap\Udomain[\uind] \not=\varnothing
    \bigr\}_{\uind\in\Uindset}
\end{equation}
is also a partial $\Ck$-atlas on $\Vdomain$.

Moreover, suppose that $\Vdomain \subset \supp\Aatlas$.
Then $\AtlasRestr{\Aatlas}{\Vdomain}$ is an \term{atlas} on $\Vdomain$, and in this case the corresponding $\Ck$-structure of $\AtlasRestr{\Aatlas}{\Vdomain}$ will be called the \term{induced $\Ck$-structure on $\Vdomain$}.
\end{subexample}

Two $\Ck$-atlases $\Aatlas$ and $\Batlas$ on $\Aspace$ are \term{$\Ck$-compatible} if their union $\Aatlas \cup \Batlas$ is still a $\Ck$-atlas, i.e.\ the charts from $\Aatlas$ and $\Batlas$ are pair-wise $\Ck$-compatible as well.

More generally, let $\Aatlas$ and $\Batlas$ be two partial $\Ck$-atlases on $\Aspace$ and $\Wdomain \subset \supp\Aatlas \cap \supp\Batlas$ be an open subset of the intersection on their supports.
Say that $\Aatlas$ and $\Batlas$ are \term{$\Ck$-compatible on $\Wdomain$}, if the atlases $\AtlasRestr{\Aatlas}{\Wdomain}$ and $\AtlasRestr{\Batlas}{\Wdomain}$ are $\Ck$-compatible.

One easily checks that the relation \term{``to be $\Ck$-compatible''} on the set of all \term{atlases} on $\Aspace$ is an equivalence relation\footnote{For \term{partial} atlases this is not true in general.}.
The corresponding equivalence classes of pair-wise $\Ck$-compatible atlases on $\Aspace$ are called \term{$\Ck$-structures}.

It also follows that each $\Ck$-structure $\Astruct$ on $\Aspace$ contains a unique \term{maximal} atlas being just a union of all atlases belonging to $\Astruct$.
If $\Aatlas$ is some $\Ck$-atlas on $\Aspace$, then the corresponding maximal $\Ck$ atlas containing $\Aatlas$ will be denoted by $\AtlasCompletion{\Aatlas}$.
Thus $\AtlasCompletion{\Aatlas}$ is a canonical representative of a $\Ck$-structure, and therefore by a \term{$\Ck$-structure} we will further mean a maximal $\Ck$-atlas.

A manifold $\Aspace$ with a fixed $\Ck$-atlas $\Aatlas$ is called a \term{$\Ck$-manifold}, and often be denoted by $\mpair{\Aspace}{\Aatlas}$.

\begin{subexample}\rm
Every non-empty open subset $\Udomain$ of every cone $\Cman\subset\bR^{n}$, for instance $\bR^{n}$ itself, has a \term{canonical} $\Ck$-structure (for each $\kk=1,\ldots,\infty$) defined by the atlas $\CanonAtlas{\Udomain}=\{(\Udomain,\id_{\Udomain})\}$ consisting of the single identity map $\id_{\Udomain}$.
We will also call that atlas \term{canonical}.
\end{subexample}

\subsection{Differentiable maps}
\label{differentiable_maps}
Let $\Aspace$ and $\Bspace$ be manifolds, $\dif\colon\Aspace\to\Bspace$ be a map, and $(\Udomain,\Ucmap)$ and $(\Vdomain,\Vcmap)$ be charts in $\Aspace$ and $\Bspace$ respectively such that $\dif(\Udomain)\subset\Vdomain$.
Then we get the following commutative diagram:
\[
\xymatrix@C=5em{
    \Udomain \ar[r]^-{\dif} \ar[d]_-{\Ucmap} & \Vdomain \ar[d]^-{\Vcmap} \\
    \Ucmap(\Udomain) \ar[r]^-{\Vcmap\circ\dif\circ\Ucmap^{-1}} & \Vcmap(\Vdomain)
}
\]
in which the lower arrow $\Vcmap\circ\dif\circ\Ucmap^{-1} \colon \Ucmap(\Udomain) \to \Vcmap(\Vdomain)$ between open subsets of some cones is called the \term{(coordinate) representation} of $\dif$ in the charts $\Uchart$ and $\Vchart$.

If $(\Udomain[1],\Ucmap[1])$ and $(\Vdomain[1],\Vcmap[1])$ is another pair of charts on $\Aspace$ and $\Bspace$ respectively such that $\dif(\Udomain[1])\subset\Vdomain[1]$ and $\Udomain[1]\cap\Udomain\ne\varnothing$, then
\begin{align}\label{equ:change_coord_repr}
    \restr{\Vcmap[1]\circ\dif\circ\Ucmap[1]^{-1}}
          {\Ucmap[1](\Udomain[1]\cap\Udomain)}
       =
    \restr{(\Vcmap[1]\circ\Vcmap^{-1})\circ(\Vcmap\circ\dif\circ\Ucmap)\circ(\Ucmap[1]\circ\Ucmap^{-1})^{-1}}
          {\Ucmap[1](\Udomain[1]\cap\Udomain)},
\end{align}
so on $\Ucmap[1](\Udomain[1]\cap\Udomain)$ the representation of $\dif$ changes by the composing from both sides with the corresponding transition maps.

Let $\Aspace$ and $\Bspace$ be $\Ck$-manifolds endowed with some maximal $\Ck$-atlases $\Aatlas$ and $\Batlas$ respectively, $\dif\colon\Aspace\to\Bspace$ be a map, and $\px\in\Aspace$.
Then $\dif$ is called \term{$\Ck$ at $\px$ \textup(with respect to $\Aatlas$ and $\Batlas$\textup)}, if there exist charts $(\Udomain,\aUcmap)\in\Aatlas$ and $(\Vdomain,\aVcmap)\in\Batlas$ such that $\px\in\Udomain$, $\dif(\Udomain)\subset\Vdomain$, and the corresponding coordinate representation $\aVcmap\circ\dif\circ\aUcmap^{-1} \colon \aUcmap(\Udomain) \to \aVcmap(\Vdomain)$ of $\dif$ is $\Ck$ at $\Ucmap(\px)$.
The identity~\eqref{equ:change_coord_repr} shows that this property does not depend on a particular choice of such charts.

Further $\dif\colon\Aspace\to\Bspace$ is called \term{$\Ck$ \textup(with respect to $\Aatlas$ and $\Batlas$\textup)} if it is so at each $\px\in\Aspace$.
In this case we will also use the notation $\dif\colon\mpair{\Aspace}{\Aatlas}\to\mpair{\Bspace}{\Batlas}$, and denote by
$\XCMAP{\kk}{}{\mpair{\Aspace}{\Aatlas},\mpair{\Bspace}{\Batlas}}$ the set of all such maps.

A homeomorphism $\dif\colon\Aspace\to\Bspace$ is a \term{$\Ck$-diffeomorphism (with respect to $\Aatlas$ and $\Batlas$)} or \term{between $\mpair{\Aspace}{\Aatlas}$ and $\mpair{\Bspace}{\Batlas}$}, if $\dif\in\XCMAP{\kk}{}{\mpair{\Aspace}{\Aatlas},\mpair{\Bspace}{\Batlas}}$ and $\dif^{-1}\in\XCMAP{\kk}{}{\mpair{\Bspace}{\Batlas},\mpair{\Aspace}{\Aatlas}}$.

The set of all $\Ck$-diffeomorphisms $\dif\colon\mpair{\Aspace}{\Aatlas}\to\mpair{\Bspace}{\Batlas}$ will be denoted by
$\XDIFF{\kk}{}{\mpair{\Aspace}{\Aatlas},\mpair{\Bspace}{\Batlas}}$.
In particular, if $\Aspace=\Bspace$ and $\Aatlas=\Batlas$, then
$\XDIFF{\kk}{}{\mpair{\Aspace}{\Aatlas},\mpair{\Aspace}{\Aatlas}}$ will also be abbreviated to
$\XDIFF{\kk}{}{\mpair{\Aspace}{\Aatlas}}$
or even to $\Diffr{\kk}{\Aspace}$ if the corresponding $\Ck$-structure on $\Aspace$ is assumed from the context.

A $\Ck$-map $\dif\colon\mpair{\Vdomain}{\Vatlas} \to \mpair{\Aspace}{\Aatlas}$ is a \term{$\Ck$-embedding}, if it is a topological embedding (homeomorphism onto its image $\dif(\Vdomain)$), and the tangent map $T_{\px}\dif\colon T_{\px}\Vdomain \to T_{\dif(\px)}\Aspace$ is injective for all $\px\in\Vdomain$.

A subset $\Udomain\subset\mpair{\Aspace}{\Aatlas}$ will be called a \term{$\Ck$-submanifold} if it is an image of some $\Ck$-embedding $\dif\colon\mpair{\Vdomain}{\Vatlas}\to\mpair{\Aspace}{\Aatlas}$.

Let us also mention two easy and useful statements.

\begin{sublemma}
Let $\mpair{\Aspace}{\Aatlas}$ be a $\Ck$-manifold, and $\Aspace \supset \Udomain \xrightarrow{\Ucmap} \Cone \subset \bR^{n}$ be a chart on $\Aspace$ not necessarily belonging to $\Aatlas$.
Denote $\Vdomain = \Ucmap(\Udomain)$, and let $\CanonAtlas{\Vdomain} = \{(\Vdomain,\id_{\Vdomain})\}$ be canonical atlas on $\Vdomain$.
Then the following conditions are equivalent:
\begin{enumerate}
\item the chart $(\Udomain,\Ucmap)$ is $\Ck$-compatible with $\Aatlas$;
\item the map $\Ucmap\colon\mpair{\Udomain}{\AtlasRestr{\Aatlas}{\Udomain}} \to \mpair{\Vdomain}{\CanonAtlas{\Vdomain}}$ is a $\Ck$-diffeomorphism;
\item the inverse map $\Ucmap^{-1}\colon\mpair{\Vdomain}{\CanonAtlas{\Vdomain}}\to\mpair{\Aspace}{\Aatlas}$ is an open $\Ck$-embedding.
\end{enumerate}
\end{sublemma}
\begin{proof}
Notice that we have the following commutative diagram
\[
\xymatrix@C=4em{
    \Udomain \ar[r]^-{\Ucmap} \ar[d]_-{\Ucmap} & \Vdomain \ar[d]^-{\id_{\Vdomain}} \\
    \Vdomain \ar[r]^-{\id_{\Vdomain}}          & \Vdomain
}
\]
showing that the coordinate representation of $\Ucmap$ with respect to the charts $(\Udomain,\Ucmap)$ and $(\Vdomain,\id_{\Vdomain})$ is the identity map $\id_{\Vdomain}$ being $\Ck$ for all $\kk=1,\ldots,\infty$.
This observation implies our lemma, and we leave the details for the reader as an exercise.
\end{proof}

\begin{sublemma}\label{lm:char_equiv_cr_structures}
Let $\Aatlas$ and $\Batlas$ be two partial $\Ck$-atlases on a manifold $\Aspace$, and
\[ \Wdomain\subset\supp\Aatlas\cap\supp\Batlas \]
be an open subset.
Then $\Aatlas$ and $\Batlas$ are $\Ck$-compatible on $\Wdomain$ if and only if the identity map $\id_{\Wdomain}$ is a $\Ck$-diffeomorphism $\mpair{\Wdomain}{\AtlasRestr{\Aatlas}{\Wdomain}}\to\mpair{\Wdomain}{\AtlasRestr{\Batlas}{\Wdomain}}$.

In particular, if $\Aatlas$ and $\Batlas$ are $\Ck$-atlases, then they define the same $\Ck$-structure on $\Aspace$ iff the identity map $\id_{\Aspace}\colon\mpair{\Aspace}{\Aatlas}\to\mpair{\Aspace}{\Batlas}$ is a $\Ck$-diffeomorphism.
\qed
\end{sublemma}

It will also be convenient to consider diffeomorphisms of pairs of manifolds.
Thus if $\mpair{\Aspace}{\Aatlas}$ and $\mpair{\Bspace}{\Batlas}$ are $\Ck$-manifolds, and $\Udomain\subset\Aspace$ and $\Vdomain\subset\Bspace$ are submanifolds, then we will denote by
$\XDIFF{\kk}{}{\mpair{(\Aspace,\Udomain)}{\Aatlas},\mpair{(\Bspace,\Vdomain)}{\Batlas}}$
the set of all $\Ck$-diffeomorphisms $\dif\colon\mpair{\Aspace}{\Aatlas}\to\mpair{\Bspace}{\Batlas}$ such that $\dif(\Udomain)=\Vdomain$.

\subsection{Classification of differentiable structures}
\label{sect:classification_of_diff_struct}
Let $\Aatlas=\{\Uchart[\uind] = (\Udomain[\uind],\Ucmap[\uind])\}_{\uind\in\Uindset}$ be a $\Ck$ atlas on $\Aspace$, and $\dif\colon\Aspace\to\Aspace$ be a homeomorphism.
Then one easily checks that
\begin{itemize}
\item the collection
\begin{equation}\label{equ:induced_atlas}
    \indatl{\dif}{\Aatlas} = \bigl\{
        \indatl{\dif}{\Uchart[\uind]} := (\dif(\Udomain[\uind]), \Ucmap[\uind]\circ\dif^{-1})
        \bigr\}_{\uind\in\Uindset}
\end{equation}
is still a $\Ck$-atlas on $\Aspace$;
\item
for every $\uind,\vind\in\Uindset$, the transition map $(\Ucmap[\vind]\circ\dif^{-1})\circ(\Ucmap[\uind]\circ\dif^{-1})^{-1}$ from the chart $\indatl{\dif}{\Uchart[\uind]}$ to $\indatl{\dif}{\Uchart[\vind]}$ coincides with the transition map $\Ucmap[\vind]\circ\Ucmap[\uind]^{-1}$ from $\Uchart[\uind]$ to $\Uchart[\vind]$;
\item the map $\dif\colon\mpair{\Aspace}{\Aatlas} \to \mpair{\Aspace}{\indatl{\dif}{\Aatlas}}$ is a $\Ck$-diffeomorphism.
\end{itemize}
The second property implies that if $\Batlas$ is another atlas on $\Aspace$ being compatible with $\Aatlas$, then $\indatl{\dif}{\Batlas}$ is compatible with $\indatl{\dif}{\Aatlas}$.
It also follows that $\Aatlas$ is maximal iff so is $\indatl{\dif}{\Aatlas}$, and thus we obtain a natural action of the group $\Homeo(\Aspace)$ of homeomorphisms of $\Aspace$ on the set of its $\Ck$-structures.

Then the problem of classification of $\Ck$-structures on $\Aspace$ can be formulated as follows:
\begin{itemize}
\item \term{describe the orbits of the action of $\Homeo(\Aspace)$ on the set of all maximal $\Ck$-atlases on $\Aspace$}.
\end{itemize}

It is well known that for every manifold $\Aspace$ of dimension $n\leq 3$ the above action is transitive.
This is usually formulated as \term{uniqueness of $\Ck$-structures} on such manifolds, for $n=1$ this is a classical result, see also discussion and elementary proof in~\cite{LysynskyiMaksymenko:SmoothStr:2024},
the case $n=2$ is proved in~\cite[Theorem~6.3]{Munkres:AnnMath:1968}, and the case $n=3$ in~\cite[Corollary~1.18]{Whitehead:AnnMath:1961}.

In particular, for the real line $\bR$, that uniqueness can be formulated as follows:
\begin{subtheorem}[Uniqueness of $\Ck$-structures on $\bR$]
\label{th:uniq_ck_struct_on_R}
Let $\mpair{\Aspace}{\Aatlas}$ be a $\Ck$-manifold, and $\Udomain\subset\Aspace$ be an open subset homeomorphic with $\bR$.
Let also $\CanonAtlas{\bR}=\{(\bR,\id_{\bR})\}$ be the canonical atlas on $\bR$.
Then there exists a $\Ck$-diffeomorphism
$\Ucmap\colon\mpair{\Udomain}{\AtlasRestr{\Aatlas}{\Udomain}}\to\mpair{\bR}{\CanonAtlas{\bR}}$.
In particular, $(\Udomain,\Ucmap)$ is a chart on $\Aspace$ being $\Ck$-compatible with $\Aatlas$.
\end{subtheorem}

\section{Differentiable structures on the non-Hausdorff letter $\Ylet$}
\label{sect:diff_struct_on_Y}

\subsection{Certain diffeomorphism groups}
\label{sect:certains_diff_groups}
In what follows we fix once and for all some $\kk\in\{1,2,\ldots,\infty\}$ and denote by $\Rpos = (0;+\infty)$ the set of positive reals.
Let also
\begin{itemize}
    \item $\DRpos$ be the group of preserving orientation $\Ck$-diffeomorphisms of $\Rpos$;
    \item $\HRplz$ be the group of homeomorphisms of $\bR$ which preserve orientation and fix $0$, which is equivalent to the requirement that those homeomorphisms leave $\Rpos$ invariant;
    \item $\DRplz$ be the subgroup of $\HRplz$ consisting of $\Ck$-diffeomorphisms.
\end{itemize}
Since $\Rpos$ is invariant under $\DRplz$, we have a natural restriction homomorphism
\[
    \rho\colon\DRplz \to \DRpos,
    \qquad
    \rho(\dif)=\erestr{\dif}{\Rpos},
\]
see Section~\ref{sect:restriction_map}.
Let also
\[ \DExtRneg := \rho(\DRplz) \subset \DRpos \]
be the image of $\rho$.
Thus it consists of $\Ck$-diffeomorphisms of $\Rpos$ which can be extended to $\Ck$-diffeomorphisms of $\bR$.

\subsection{Two actions on the \term{set} $\DRpos$}
As mentioned above, there is a left action of $\DExtRneg\wr\bZ_{2}$ on $\DRpos$ defined by the rule:
\[
(\hadif,\hbdif, \delta)\cdot \gdif := (\hbdif\circ\gdif\circ\hadif^{-1})^{\delta},
\]
for $\hadif,\hbdif\in\DExtRneg$, $\delta\in\{\pm1\}$, and $\gdif\in\DRpos$.
The corresponding set of orbits will be denoted by
\[
\dbli{\DExtRneg}{\DRpos}
\]
and called \term{$(\DExtRneg,\pm)$-double cosets}%
\footnote{\,This term is not standard and was used in~\cite{LysynskyiMaksymenko:SmoothStr:2024}}.

Consider the restriction of the above action to the subgroup $\DExtRneg\times\DExtRneg\times 1$:
\[
(\hadif,\hbdif, 1)\cdot \gdif := \hbdif\circ\gdif\circ\hadif^{-1},
\]
$\hadif,\hbdif\in\DExtRneg$, $\gdif\in\DRpos$.
The corresponding orbits set is denoted by
\[
\dbl{\DExtRneg}{\DRpos}{\DExtRneg}
\]
and called \term{$\DExtRneg$-double cosets}.

Our aim is to show that these double cosets classify $\Ck$-structures of the non-Hausdorff letter $\Ylet$, see Theorem~\ref{th:Ck_str_on_Y} below.

\subsection{Non-Hausdorff letter $\Ylet$}
Let
\[
    \Ylet = (\bR\times\{0,1\})/\{ (x,0)\sim(x,1) \ \text{for} \ x\in\Rpos\}.
\]
be the topological space obtained by gluing two copies of $\bR$ via the identity homeomorphism of $\Rpos$, see Figure~\ref{fig:LY}.
\begin{figure}[htbp!]
\includegraphics[height=2.5cm]{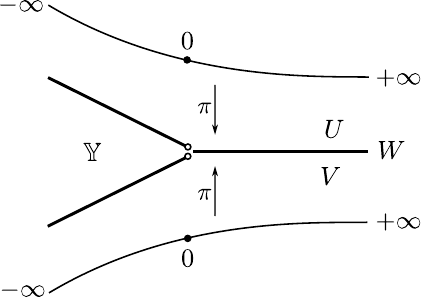}
\caption{Manifold $\Ylet$}\label{fig:y_letter}
\end{figure}
Let also $\pi\colon\RTwoCopies \to \Ylet$ be the corresponding quotient map.
Denote
\begin{gather*}
    \Udomain = \pi(\bR\times\{0\}),\qquad
    \Vdomain = \pi(\bR\times\{1\}),\\
    \Wdomain = \Udomain \cap \Vdomain
    = \pi\bigl(\Rpos \times\{0\}\bigr)
    = \pi\bigl(\Rpos \times\{1\}\bigr).
\end{gather*}
It is easy to see that every homeomorphism $\dif\colon\Ylet\to\Ylet$ either leaves $\Udomain$ and $\Vdomain$ invariant or exchanges them.
In particular, $\dif$ always preserves $\Wdomain$.

It is also evident that the restriction of $\pi$ onto $\bR\times\{0\}$ and $\bR\times\{1\}$ are open embeddings, whence the inverse maps
\begin{equation}\label{equ:canonical_atlas_on_Y}
\begin{aligned}
    \canUcmap &= (\restr{\pi}{\bR\times\{0\}})^{-1}\colon\Udomain\to\bR\times\{0\} \equiv \bR, &
    \canVcmap &= (\restr{\pi}{\bR\times\{1\}})^{-1}\colon\Vdomain\to\bR\times\{1\} \equiv \bR,
\end{aligned}
\end{equation}
can be regarded as charts $(\Udomain,\canUcmap)$ and $(\Vdomain,\canVcmap)$ of $\Ylet$.
Note that the corresponding transition map $\canVcmap\circ\canUcmap^{-1} = \id_{\Rpos}$, whence the atlas $\CanonAtlas{\Ylet}=\{ (\Udomain,\canUcmap), (\Vdomain,\canVcmap) \}$ on $\Ylet$ is $\Ck$ for all $\kk=1,\ldots,\infty$.
We will call this atlas \term{canonical}.

\begin{subdefinition}
A $\Ck$-atlas $\Aatlas$ on $\Ylet$ will be called \term{minimal} if it consists of two charts $\Ucmap\colon\Udomain\to\bR$ and $\Vcmap\colon\Vdomain\to\bR$ being homeomorphisms onto and such that
\begin{equation}\label{equ:min_atlas__uW_vW_pos}
    \Ucmap(\Wdomain)=\Vcmap(\Wdomain)=\Rpos.
\end{equation}
\end{subdefinition}

For instance, the above canonical atlas $\CanonAtlas{\Ylet}$ is minimal.

Notice also that the transition map $\Vcmap\circ\Ucmap^{-1}$ of a minimal $\Ck$-atlas preserves orientation of $\Rpos$, and thus belongs to $\DRpos$.
The following lemma shows that every $\gdif\in\DRpos$ is a transition map for some minimal $\Ck$-atlas of $\Ylet$.

\begin{sublemma}\label{lm:min_atlas_realization_on_Y}
Let $\hat{\Vcmap}\colon\overline{\Vdomain\setminus\Wdomain} \to(-\infty;0]$ and $\Ucmap\colon\Udomain\to\bR$ be two homeomorphisms such that $\Ucmap(\Wdomain)=\Rpos$.
Then for any $\Ck$-diffeomorphism $\gdif\in\DRpos$, the homeomorphism $\hat{\Vcmap}$ extends to a unique homeomorphism $\Vcmap\colon\Vdomain\to\bR$ such that the pair $\Aatlas=\{(\Udomain,\Ucmap),(\Vdomain,\Vcmap)\}$ constitutes a minimal $\Ck$-atlas on $\Ylet$ whose transition map $\Vcmap\circ\Ucmap^{-1}$ coincides with $\gdif$.
\end{sublemma}
\begin{proof}
One easily checks $\Vcmap\colon\Vdomain\to\bR$ should be given by
\[
\Vcmap(\px) =
\begin{cases}
\hat{\Vcmap}(\px),     & \px\in\overline{\Vdomain\setminus\Wdomain}, \\
\gdif\circ\Ucmap(\px), & \px\in\Wdomain.
\end{cases}
\qedhere
\]
\end{proof}

In particular, let $\CanonAtlas{\Ylet}=\{ (\Udomain,\canUcmap), (\Vdomain,\canVcmap) \}$ be the canonical atlas on $\Ylet$.
For each $\gdif\in\DRpos$ define the homeomorphism $\Vcmap[\gdif]\colon\Vdomain\to\bR$ by
\[
    \Vcmap[\gdif](\px) =
    \begin{cases}
        \canVcmap(\px),             & \px \leq 0, \\
        \gdif\circ\canUcmap(\px),   & \px > 0.
    \end{cases}
\]
Then the pair $\Aatlas_{\gdif} = \{ (\Udomain,\canUcmap), (\Vdomain,\Vcmap[\gdif]) \}$ is a $\Ck$-atlas on $\Ylet$ whose transition map is $\gdif$, while the correspondence $\gdif\mapsto\Aatlas_{\gdif}$ is an injection of $\DRpos$ into the set of all minimal atlases on $\Ylet$.

\begin{sublemma}\label{lm:min_atlas_on_Y}
Every $\Ck$-structure on $\Ylet$ has a minimal $\Ck$-atlas.
\end{sublemma}
\begin{proof}
Let $\Batlas$ be any $\Ck$-atlas on $\Ylet$.
We should construct a $\Ck$-compatible minimal atlas $\Aatlas$ on $\Ylet$.

Let $\CanonAtlas{\bR}=\{ (\bR,\id_{\bR})\}$ be the canonical atlas on $\bR$.
Since $\Udomain$ and $\Vdomain$ are open subsets of $\Ylet$ homeomorphic with $\bR$, it follows from Theorem~\ref{th:uniq_ck_struct_on_R}, that there exist $\Ck$-diffeomorphisms
\begin{align*}
    &\Ucmap\colon\mpair{\Udomain}{\AtlasRestr{\Batlas}{\Udomain}}\to\mpair{\bR}{\CanonAtlas{\bR}},&
    &\Vcmap\colon\mpair{\Vdomain}{\AtlasRestr{\Batlas}{\Vdomain}}\to\mpair{\bR}{\CanonAtlas{\bR}}.
\end{align*}
One can also assume that~\eqref{equ:min_atlas__uW_vW_pos} is satisfied.
Indeed, we have that $\Ucmap(\Wdomain)$ is either $(-\infty;a)$ or $(a;+\infty)$ for some $a\in\bR$.
Take any $\Ck$-diffeomorphism $p\colon\bR\to\bR$ such that $p(\Ucmap(\Wdomain)) = \Rpos$ and replace $\Ucmap$ with $p\circ\Ucmap$.
Similar changes should be made for $\Vcmap$ if necessary.

Then we get two charts $(\Udomain,\Ucmap)$ and $(\Vdomain,\Vcmap)$ on $\Ylet$ being $\Ck$-compatible with $\Batlas$ and satisfying~\eqref{equ:min_atlas__uW_vW_pos}.
Since $\Ylet=\Udomain\cup\Vdomain$, the pair $\Aatlas=\{(\Udomain,\Ucmap),(\Vdomain,\Vcmap)\}$ is a minimal $\Ck$-atlas on $\Ylet$ being $\Ck$-compatible with $\Batlas$, and thus defining the same $\Ck$-structure.
\end{proof}

Denote by $\mathfrak{C}^{\kk}(\Ylet)$ the set of all $\Ck$-structures on $\Ylet$.
Let also $\Homeo(\Ylet)$ be the group of all homeomorphisms of $\Ylet$, and $\Homeo(\Ylet,\Udomain,\Vdomain)$ be its subgroup leaving invariant $\Udomain$ and $\Vdomain$.
The following statement classifies $\Ck$-structures on $\Ylet$.

\begin{subtheorem}\label{th:Ck_str_on_Y}
The correspondence
\[
\Aatlas=\{(\Udomain,\Ucmap),(\Vdomain,\Vcmap)\}
\ \longmapsto \
\Atrmap = \Vcmap^{-1}\circ\Ucmap
\]
associating to each minimal $\Ck$-atlas $\Aatlas$ on $\Ylet$ its transition map $\Atrmap$ yields bijections between
\begin{itemize}
\item
the set $\mathfrak{C}^{\kk}(\Ylet) / \Homeo(\Ylet)$ of all $\Ck$-structures on $\Ylet$ up to a $\Ck$-diffeomorphism and
\item
the set $\dbli{\DExtRneg}{\DRpos}$,
\end{itemize}
and also between
\begin{itemize}
\item
the set $\mathfrak{C}^{\kk}(\Ylet) / \Homeo(\Ylet,\Udomain,\Vdomain)$ of $\Ck$-structures on $\Ylet$ up to a $\Ck$-diffeomorphism leaving invariant $\Udomain$ and $\Vdomain$ and
\item
the set $\dbl{\DExtRneg}{\DRpos}{\DExtRneg}$.
\end{itemize}
\end{subtheorem}

For the proof is based on the following lemma.
\begin{sublemma}\label{lm:diffs_of_Y}
Let $\Aatlas=\{(\Udomain,\aUcmap),(\Vdomain,\aVcmap)\}$ and $\Batlas=\{(\Udomain,\bUcmap),(\Vdomain,\bVcmap)\}$ be two minimal $\Ck$-atlases on $\Ylet$, and $\Atrmap=\aVcmap\circ\aUcmap^{-1}, \Btrmap=\bVcmap\circ\bUcmap^{-1} \in \DRpos$ be their respective transition maps.
\begin{enumerate}[leftmargin=*, itemsep=1ex, label={\rm\arabic*)}]
\item\label{enum:lm:diffs_of_Y:AB_compatible}
Then $\Aatlas$ and $\Batlas$ are $\Ck$-compatible iff $\bUcmap\circ\aUcmap^{-1}$ and $\bVcmap\circ\aVcmap^{-1}\in\DRplz$, i.e.\ $\bUcmap$ and $\aUcmap$ (as well as $\bVcmap$ and $\aVcmap$) belong to the same adjacent right class $\HRplz/\DRplz$.

\item\label{enum:lm:diffs_of_Y:UV}
The following conditions are equivalent:
\begin{enumerate}[label={\rm(2\alph*)}]
\item\label{enum:lm:diffs_of_Y:UV:h}
there exists a $\Ck$-diffeomorphism $\dif\colon\mpair{\Ylet}{\Aatlas}\to\mpair{\Ylet}{\Batlas}$ leaving invariant $\Udomain$ and $\Vdomain$;
\item\label{enum:lm:diffs_of_Y:UV:gb__b_ga_ainv}
there exist $\hadif,\hbdif\in\DExtRneg$ such that $\Btrmap = \hbdif\circ \Atrmap \circ \hadif^{-1}$.
\end{enumerate}

\item\label{enum:lm:diffs_of_Y:VU}
Similarly, the following conditions are also equivalent:
\begin{enumerate}[label={\rm(3\alph*)}]
\item\label{enum:lm:diffs_of_Y:VU:h}
there exists a $\Ck$-diffeomorphism $\dif\colon\mpair{\Ylet}{\Aatlas}\to\mpair{\Ylet}{\Batlas}$ exchanging $\Udomain$ and $\Vdomain$;
\item\label{enum:lm:diffs_of_Y:VU:gbinv__b_ga_ainv}
there exist $\hadif,\hbdif\in\DExtRneg$ such that $\Btrmap^{-1} = \hbdif \circ \Atrmap \circ \hadif^{-1}$.
\end{enumerate}
\end{enumerate}
\end{sublemma}
\begin{proof}
\ref{enum:lm:diffs_of_Y:AB_compatible}
Notice that $\adif:=\bUcmap\circ\aUcmap^{-1}$ and $\bdif:=\bVcmap\circ\aVcmap^{-1}$ are the transition maps between the respective charts of $\Aatlas$ and $\Batlas$.
Since all coordinate homeomorphisms send $\Wdomain$ onto $\Rpos$, it follows that $\adif$ and $\bdif$ preserve $\Rpos$, i.e.\ they belong to $\DRplz$.
It remains to note that $\Aatlas$ and $\Batlas$ are $\Ck$-compatible iff $\adif,\bdif\in\DRplz$.

\ref{enum:lm:diffs_of_Y:UV}
Let $\dif\colon\Ylet\to\Ylet$ be a homeomorphism which leaves invariant $\Udomain$ and $\Vdomain$, i.e.\ $\dif(\Udomain)=\Udomain$ and $\dif(\Vdomain)=\Vdomain$.
Then its coordinate representation with respect to the pair of charts $(\Udomain,\aUcmap)$ and $(\Udomain,\bUcmap)$ and also with respect to another pair of charts $(\Vdomain,\aVcmap)$ and $(\Vdomain,\bVcmap)$ can be seen from the following diagram:
\begin{equation*}
\begin{aligned}
    &
    \xymatrix@C=6em{
        \Udomain \ar[r]^{\erestr{\dif}{\Udomain}} \ar[d]_{\aUcmap}                      &
        \Udomain \ar[d]^{\bUcmap} \\
        \bR \ar[r]^{\adif = \bUcmap\circ\erestr{\dif}{\Udomain}\circ\aUcmap^{-1}}  &
        \bR
    }
    &
    \qquad\qquad
    &
    \xymatrix@C=6em{
        \Vdomain \ar[r]^{\erestr{\dif}{\Vdomain}} \ar[d]_{\aVcmap}                      &
        \Vdomain \ar[d]^{\bVcmap}\\
        \bR \ar[r]^{\bdif = \bVcmap\circ\erestr{\dif}{\Vdomain}\circ\aVcmap^{-1}}  &
        \bR
    }
\end{aligned}
\end{equation*}

Since $\dif(\Wdomain)=\Wdomain$, it follows from~\eqref{equ:min_atlas__uW_vW_pos} that $\adif$ and $\bdif$ are homeomorphisms of $\bR$ which leave invariant $\Rpos$, i.e.\ they belong to $\HRplz$.
Let us join the above pair of diagrams including also the maps of $\Wdomain$:
\begin{equation}\label{equ:h_UU_VV__Y}
\begin{gathered}
\xymatrix@C=3em@R=1.5em{
    \bR
        \ar@{-->}@/_5ex/[ddd]_-{\adif}
    &
    \Rpos
        \ar@{_(->}[l]
        \ar[rr]^-{\Atrmap \,=\, \aVcmap\,\circ\,\aUcmap^{-1}}
        \ar@{-->}@/^3.5ex/[ddd]^-{\erestr{\adif}{\Rpos}}
    &&
    \Rpos
        \ar@{^(->}[r]
        \ar@{-->}@/_3.5ex/[ddd]_-{\erestr{\bdif}{\Rpos}}
    &
    \bR
        \ar@{-->}@/^5ex/[ddd]^-{\bdif}
    \\
    \Udomain
        \ar[u]_-{\aUcmap}
        \ar[d]_-{\erestr{\dif}{\Udomain}}
    &
    \Wdomain
        \ar@{_(->}[l]
        \ar@{=}[rr]
        \ar[u]^-{\aUcmap}
        \ar[d]_-{\erestr{\dif}{\Wdomain}}
    &&
    \Wdomain
        \ar@{^(->}[r]
        \ar[u]_-{\aVcmap}
        \ar[d]^-{\erestr{\dif}{\Wdomain}}
    &
    \Vdomain
        \ar[u]^-{\aVcmap}
        \ar[d]^-{\erestr{\dif}{\Vdomain}}
    \\
    \Udomain
        \ar[d]^-{\bUcmap}
    &
    \Wdomain
        \ar@{_(->}[l]
        \ar@{=}[rr]
        \ar[d]_-{\bUcmap}
    &&
    \Wdomain
        \ar@{^(->}[r]
        \ar[d]^-{\bVcmap}
    &
    \Vdomain
        \ar[d]_-{\bVcmap}
    \\
    \bR
    &
    \Rpos
        \ar@{_(->}[l]
        \ar[rr]^-{\Btrmap \,=\, \bVcmap\,\circ\,\bUcmap^{-1}}
    &&
    \Rpos
        \ar@{^(->}[r]
    &
    \bR
}
\end{gathered}
\end{equation}
Denote $\hadif=\erestr{\adif}{\Rpos}$ and $\hbdif=\erestr{\bdif}{\Rpos}$.
Then the commutativity of the middle vertical rectangle gives the identity $\Btrmap = \hbdif \circ \Atrmap \circ \hadif^{-1}$ on $\Rpos$ required in~\ref{enum:lm:diffs_of_Y:UV:gb__b_ga_ainv}.

Notice also that $\dif$ is a $\Ck$-diffeomorphism iff $\adif,\bdif$ are $\Ck$-diffeomorphisms of $\bR$.
In that case $\adif,\bdif\in\DRplz$, and therefore $\hadif,\hbdif\in\DExtRneg$.
This proves the implication \ref{enum:lm:diffs_of_Y:UV:h}$\Rightarrow$\ref{enum:lm:diffs_of_Y:UV:gb__b_ga_ainv}.

Conversely, suppose $\Btrmap = \hbdif \circ \Atrmap \circ \hadif^{-1}$ for certain $\hadif,\hbdif\in\DExtRneg$, so $\hadif=\erestr{\adif}{\Rpos}$ and $\hbdif=\erestr{\bdif}{\Rpos}$ for some $\adif,\bdif\in\DRplz$.
One easily checks that then the following map
\begin{equation}\label{equ:h_Y_Y__reconstr_via_a_b}
\dif\colon\Ylet\to\Ylet,
\qquad
\dif(\px) =
\begin{cases}
\bUcmap^{-1}\circ\adif\circ\aUcmap(\px), & \px\in\Udomain, \\
\bVcmap^{-1}\circ\bdif\circ\aVcmap(\px), & \px\in\Vdomain.
\end{cases}
\end{equation}
is a well-defined homeomorphism of $\Ylet$ leaving invariant both $\Udomain$ and $\Vdomain$.
Moreover, its respective coordinate representations are $\adif$ and $\bdif$, whence $\dif$ is a $\Ck$-diffeomorphism of $\Ylet$.
This proves the inverse implication \ref{enum:lm:diffs_of_Y:UV:gb__b_ga_ainv}$\Rightarrow$\ref{enum:lm:diffs_of_Y:UV:h}.

\ref{enum:lm:diffs_of_Y:VU}
This part is similar.
Let $\dif\colon\Ylet\to\Ylet$ be a homeomorphism which exchanges $\Udomain$ and $\Vdomain$, i.e.\ $\dif(\Udomain)=\Vdomain$ and $\dif(\Vdomain)=\Udomain$.
Then we have the following commutative diagram in which the last two rows are reversed in comparison with~\eqref{equ:h_UU_VV__Y}:
\begin{equation}\label{equ:h_UV_VU__Y}
\begin{gathered}
\xymatrix@C=2.1em@R=1.5em{
    \bR
        \ar@{-->}@/_4ex/[ddd]_-{\adif}
    &
    \Rpos
        \ar@{_(->}[l]
        \ar[rr]^-{\Atrmap = \aVcmap \circ \aUcmap^{-1}}
        \ar@{-->}@/^2.5ex/[ddd]^-{\adif}
    &&
    \Rpos
        \ar@{^(->}[r]
        \ar@{-->}@/_2.5ex/[ddd]_-{\bdif}
    &
    \bR
        \ar@{-->}@/^4ex/[ddd]^-{\bdif}
    \\
    \Udomain
        \ar[u]_-{\aUcmap}
        \ar[d]_-{\dif}
    &
    \Wdomain
        \ar@{_(->}[l]
        \ar@{=}[rr]
        \ar[u]^-{\aUcmap}
        \ar[d]_-{\dif}
    &&
    \Wdomain
        \ar@{^(->}[r]
        \ar[u]_-{\aVcmap}
        \ar[d]^-{\dif}
    &
    \Vdomain
        \ar[u]^-{\aVcmap}
        \ar[d]^-{\dif}
    \\
    \Vdomain
        \ar[d]^-{\bVcmap}
    &
    \Wdomain
        \ar@{_(->}[l]
        \ar@{=}[rr]
        \ar[d]_-{\bVcmap}
    &&
    \Wdomain
        \ar@{^(->}[r]
        \ar[d]^-{\bUcmap}
    &
    \Udomain
        \ar[d]_-{\bUcmap}
    \\
    \bR
    &
    \Rpos
        \ar@{_(->}[l]
    &&
    \Rpos
        \ar@{^(->}[r]
        \ar[ll]_-{\, \Btrmap = \bVcmap\circ\bUcmap^{-1}}
    &
    \bR
}
\end{gathered}
\end{equation}
Here $\adif=\bVcmap\circ\dif\circ\aUcmap^{-1}$ is a coordinate representation of $\dif$ with respect to the charts $(\Udomain,\aUcmap)$ and $(\Vdomain,\bVcmap)$, while $\bdif=\bUcmap\circ\dif\circ\aVcmap^{-1}$ is a coordinate representation of $\dif$ with respect to the charts $(\Vdomain,\aVcmap)$ and $(\Udomain,\bUcmap)$.
They belong to $\HRplz$, and if we denote $\hadif = \erestr{\adif}{\Rpos}$ and $\hbdif = \erestr{\bdif}{\Rpos}$, then the commutativity of the middle vertical rectangle of~\eqref{equ:h_UV_VU__Y} means the identity $\Btrmap^{-1} = \hbdif \circ \Atrmap \circ \hadif^{-1}$ from~\ref{enum:lm:diffs_of_Y:VU:gbinv__b_ga_ainv}.

Again, $\dif$ is a $\Ck$-diffeomorphism iff $\adif$ and $\bdif$ are diffeomorphisms of $\bR$.
In that case $\hadif, \hbdif \in \DExtRneg$, which proves the implication \ref{enum:lm:diffs_of_Y:VU:h}$\Rightarrow$\ref{enum:lm:diffs_of_Y:VU:gbinv__b_ga_ainv}.
The proof of the inverse implication is similar to~\ref{enum:lm:diffs_of_Y:UV} and we leave it for the reader.
\end{proof}

As a consequence we get an explicit description of diffeomorphism groups between $\Ck$-structures on $\Ylet$ and a classification of such structures.

\begin{subcorollary}
\newcommand\DYAB{\XDIFF{\kk}{}{\mpair{\Ylet}{\Aatlas},\mpair{\Ylet}{\Batlas}}}
\newcommand\DYAA{\XDIFF{\kk}{}{\mpair{\Ylet}{\Aatlas}}}
Consider the following subset of $\DRplz\wr\bZ_{2}$:
\[
    E = \{ (\adif,\bdif,\delta) \in \DRplz\wr\bZ_{2} \mid
    \Btrmap = (\erestr{\bdif}{\Rpos} \circ \Atrmap \circ \erestr{\adif^{-1}}{\Rpos})^{\delta}
    \}.
\]
Then the correspondence
$\eta\colon\DYAB\to E$,
\begin{equation}\label{equ:bijection_DiffYAYB_E}
    \eta(\dif) =
    \begin{cases}
    \bigl(
        \bUcmap\circ\erestr{\dif}{\Udomain}\circ\aUcmap^{-1}, \
        \bVcmap\circ\erestr{\dif}{\Vdomain}\circ\aVcmap^{-1}, \
        1
    \bigr), &
    \text{if} \
    \dif(\Udomain)=\Udomain,
    \\
    \bigl(
        \bVcmap\circ\erestr{\dif}{\Udomain}\circ\aUcmap^{-1}, \
        \bUcmap\circ\erestr{\dif}{\Vdomain}\circ\aVcmap^{-1}, \
        -1
    \bigr), &
    \text{if} \
    \dif(\Udomain)=\Vdomain,
    \end{cases}
\end{equation}
associating to each $\dif\in\DYAB$ the pair of its coordinate representations together with a sign corresponding to the permutation of $\Udomain$ and $\Vdomain$ is a bijection.

If $\Aatlas=\Batlas$, then $E$ is a subgroup of $\DRplz\wr\bZ_{2}$ and the bijection~\eqref{equ:bijection_DiffYAYB_E}
\begin{equation}\label{equ:bijection_DiffYAYA_E}
    \eta(\dif) =
    \begin{cases}
    \bigl(
        \aUcmap\circ\erestr{\dif}{\Udomain}\circ\aUcmap^{-1}, \
        \aVcmap\circ\erestr{\dif}{\Vdomain}\circ\aVcmap^{-1}, \
        1
    \bigr), &
    \text{if} \
    \dif(\Udomain)=\Udomain,
    \\
    \bigl(
        \aVcmap\circ\erestr{\dif}{\Udomain}\circ\aUcmap^{-1}, \
        \aUcmap\circ\erestr{\dif}{\Vdomain}\circ\aVcmap^{-1}, \
        -1
    \bigr), &
    \text{if} \
    \dif(\Udomain)=\Vdomain,
    \end{cases}
\end{equation}
becomes an isomorphism of groups $\eta\colon\DYAA\to E$.
\qed
\end{subcorollary}

\begin{proof}[Proof of Theorem~\ref{th:Ck_str_on_Y}]
To simplify notations denote $\calE:=\DExtRneg$.
For a maximal $\Ck$-atlas $\Astruct$ (i.e.\ a $\Ck$-structure) on $\Ylet$ let
\begin{itemize}[leftmargin=*]
\item $[\Astruct]$ be its equivalence class in $\mathfrak{C}^{\kk}(\Ylet) / \Homeo(\Ylet)$, i.e.\ the set of all $\Ck$-structures on $\Ylet$ being $\Ck$-diffeomorphic to $\mpair{\Ylet}{\Astruct}$ via a $\Ck$-diffeomorphism;

\item and $[\Astruct]'$ be its equivalence class in $\mathfrak{C}^{\kk}(\Ylet) / \Homeo(\Ylet,\Udomain,\Vdomain)$, i.e.\ the set of all $\Ck$-structures on $\Ylet$ being $\Ck$-diffeomorphic to $\mpair{\Ylet}{\Astruct}$ via a $\Ck$-diffeomorphism leaving invariant $\Udomain$ and $\Vdomain$.
\end{itemize}

Now let $\Astruct$ and $\Bstruct$ be two maximal $\Ck$-atlases on $\Ylet$,
\begin{align*}
    \Aatlas&=\{(\Udomain,\aUcmap),(\Vdomain,\aVcmap)\} \subset \Astruct, &
    \Batlas&=\{(\Udomain,\bUcmap),(\Vdomain,\bVcmap)\} \subset \Bstruct
\end{align*}
be some minimal subatlases, and $\Atrmap=\aVcmap\circ\aUcmap^{-1}, \Btrmap=\bVcmap\circ\bUcmap^{-1} \in \DRpos$ be their respective transition maps.

1) Suppose there exists a homeomorphism $\dif$ of $\Ylet$ leaving invariant $\Udomain$ and $\Vdomain$ and being also a $\Ck$-diffeomorphism $\dif\colon\mpair{\Ylet}{\Astruct}\to\mpair{\Ylet}{\Bstruct}$.
Then it is also a $\Ck$-diffeomorphism $\dif\colon\mpair{\Ylet}{\Aatlas}\to\mpair{\Ylet}{\Batlas}$, whence, by the implication \ref{enum:lm:diffs_of_Y:UV:h}$\Rightarrow$\ref{enum:lm:diffs_of_Y:UV:gb__b_ga_ainv} of Lemma~\ref{lm:diffs_of_Y}\ref{enum:lm:diffs_of_Y:UV:gb__b_ga_ainv}, $\Btrmap = \hbdif\circ \Atrmap \circ \hadif^{-1}$ for some $\hadif,\hbdif\in\calE$.
Therefore $\Btrmap$ and $\Atrmap$ belong to the same $\calE$-double coset.
Hence that coset depends on $\Astruct$ only and we will denote it by $\gdif_{\Astruct}'$.
Thus we get a well-defined correspondence
\[
   \mathfrak{C}^{\kk}(\Ylet) / \Homeo(\Ylet,\Udomain,\Vdomain)
   \ \ni \ [\Astruct]' \  \xrightarrow{~~~\mu'~~~} \ \gdif'_{\Astruct} \ \in \
   \dbl{\calE}{\DRpos}{\calE}.
\]
Note that if $\Btrmap = \hbdif\circ \Atrmap \circ \hadif^{-1}$ for some $\hadif,\hbdif\in\calE$, then by the inverse implication~\ref{enum:lm:diffs_of_Y:UV:gb__b_ga_ainv}$\Rightarrow$\ref{enum:lm:diffs_of_Y:UV:h} the structures $\Astruct$ and $\Bstruct$ are $\Ck$-diffeomorphic, i.e.\ $\mu'$ is injective.

Moreover, by Lemma~\ref{lm:min_atlas_realization_on_Y}, every $\gdif\in\DRpos$ is a transition map for some minimal $\Ck$-atlas on $\Ylet$, which means that $\mu'$ is surjective as well.

2) Suppose there exists a homeomorphism $\dif$ of $\Ylet$ being a $\Ck$-diffeomorphism $\dif\colon\mpair{\Ylet}{\Astruct}\to\mpair{\Ylet}{\Bstruct}$.
Then by the implications~\ref{enum:lm:diffs_of_Y:UV:h}$\Rightarrow$\ref{enum:lm:diffs_of_Y:UV:gb__b_ga_ainv} and~\ref{enum:lm:diffs_of_Y:VU:h}$\Rightarrow$\ref{enum:lm:diffs_of_Y:VU:gbinv__b_ga_ainv} of Lemma~\ref{lm:diffs_of_Y} there exist $\hadif,\hbdif\in\calE$ such that $\hbdif\circ \Atrmap \circ \hadif^{-1}$ coincides either with $\Btrmap$ or with $\Btrmap^{-1}$.
Hence $\Btrmap$ and $\Atrmap$ belong to the same $(\calE,\pm)$-double coset, which we will denote by $\gdif_{\Astruct}$.
In particular, we obtain a well-defined correspondence
\[
   \mathfrak{C}^{\kk}(\Ylet) / \Homeo(\Ylet)
   \ \ni \ [\Astruct] \  \xrightarrow{~~~\mu~~~}  \ \gdif_{\Astruct} \ \in \
   \dbli{\calE}{\DRpos}.
\]

Then again the inverse implications~\ref{enum:lm:diffs_of_Y:UV:gb__b_ga_ainv}$\Rightarrow$\ref{enum:lm:diffs_of_Y:UV:h} and~\ref{enum:lm:diffs_of_Y:VU:gbinv__b_ga_ainv}$\Rightarrow$\ref{enum:lm:diffs_of_Y:VU:h} guarantees that $\mu$ is injective, while Lemma~\ref{lm:min_atlas_realization_on_Y} implies surjectivity of $\mu$.
\end{proof}

\subsection{Examples of $\Ck$-structures on $\Ylet$}
In this subsection we show that $\Ylet$ admits a continuum of pairwise non-diffeomorphic $\Ck$-structures for all $\kk=1,\ldots,\infty$.
Moreover, there are diffeomorphic $\Ck$-structures all diffeomorphisms between them always sent $\Udomain$ to $\Udomain$.
Also there are diffeomorphic $\Ck$-structures all diffeomorphisms between them always sent $\Udomain$ to $\Vdomain$.

Our first statement is a characterization of $\Ck$-structures on $\Ylet$ diffeomorphic to the canonical one which is defined by the canonical atlas $\CanonAtlas{\Ylet}=\{ (\Udomain,\canUcmap), (\Vdomain,\canVcmap) \}$, see~\eqref{equ:canonical_atlas_on_Y}.

\begin{sublemma}\label{lm:Y:diff_to_canonical_structure}
Let $\Aatlas=\{(\Udomain,\Ucmap),(\Vdomain,\Vcmap)\}$ be any minimal $\Ck$-atlas on $\Ylet$.
Then $\mpair{\Ylet}{\Aatlas}$ is $\Ck$-diffeomorphic to $\mpair{\Ylet}{\CanonAtlas{\Ylet}}$ iff the transition map $\Atrmap=\Vcmap\circ\Ucmap^{-1} \in \DExtRneg$, i.e.\ it extends to a diffeomorphism of all $\bR$.
\end{sublemma}
\begin{proof}
Recall that the transition map of $\CanonAtlas{\Ylet}$ is $\id_{\Rpos}$.

\emph{Necessity.}
Suppose there is a $\Ck$-diffeomorphism $\dif\colon\mpair{\Ylet}{\CanonAtlas{\Ylet}}\to\mpair{\Ylet}{\Aatlas}$.
Then by Lemma~\ref{lm:diffs_of_Y}, $\Atrmap = (\bdif\circ\id_{\Rpos}\circ\adif^{-1})^{\delta}$ for some $\adif,\bdif\in\DExtRneg$ and $\delta\in\{\pm1\}$.
In particular, $\Atrmap\in\DExtRneg$ as well.

\emph{Sufficiency.}
Suppose $\Atrmap\in\DExtRneg$, so $\Atrmap = \erestr{\bdif}{\Rpos}$ for some $\bdif\in\DRplz$.
Let $\adif = \id_{\bR}$.
Then we can also write $\Atrmap = \erestr{\bdif}{\Rpos} \circ \id_{\bR} \circ \erestr{\adif^{-1}}{\Rpos}$, whence by~\eqref{equ:h_Y_Y__reconstr_via_a_b}, the map $\dif\colon\Ylet\to\Ylet$,
\[
\dif\colon\Ylet\to\Ylet,
\qquad
\dif(\px) =
\begin{cases}
\tilde{\Ucmap}^{-1}\circ\adif\circ\Ucmap(\px), & \px\in\Udomain, \\
\tilde{\Vcmap}^{-1}\circ\bdif\circ\Vcmap(\px), & \px\in\Vdomain
\end{cases}
\]
is a $\Ck$-diffeomorphism $\mpair{\Ylet}{\CanonAtlas{\Ylet}} \to \mpair{\Ylet}{\Aatlas}$.
\end{proof}

\begin{subexample}\rm
\newcommand\jdif[1]{\phi_{#1}}
For every $a>0$ define the following diffeomorphism $\jdif{a}\in\DRpos$, $\jdif{a}(\px)=a\px$.
Let also $\Aatlas$ be a minimal $\Ck$-atlas on $\Ylet$ whose transition map is $\jdif{a}$.
Since $\jdif{a}$ extends to a diffeomorphism of all of $\Ylet$, we have by Lemma~\ref{lm:Y:diff_to_canonical_structure} that $\mpair{\Ylet}{\Aatlas}$ is $\Ck$-diffeomorphic with $\mpair{\Ylet}{\CanonAtlas{\Ylet}}$.
\end{subexample}

\newcommand\jdif[1]{g_{#1}}
\newcommand\vs{s}
\newcommand\vt{t}

\begin{sublemma}\label{lm:x_pow_alpha}
For every $\vs>0$ define the following diffeomorphism $\jdif{\vs}\in\DRpos$, $\jdif{\vs}(\px) = \px^{\vs}$, and for simplicity denote $\calE := \DExtRneg$.
Then for $\vs,\vt>0$ the following statements hold.
\begin{enumerate}[leftmargin=*]
\item\label{enum:lm:x_pow_alpha:EDE}
$\jdif{\vs}$ and $\jdif{\vt}$ belong to the same $\calE$-double coset of $\dbl{\calE}{\DRpos}{\calE}$ if and only if $\vs=\vt$.

\item\label{enum:lm:x_pow_alpha:EDpmE}
$\jdif{\vs}$ and $\jdif{\vt}$ belong to the same $(\calE,\pm)$-double coset of $\dbli{\calE}{\DRpos}$ if and only if either $\vs=\vt$ or $\vs = 1/\vt$.
\end{enumerate}
\end{sublemma}
\begin{proof}
Note that each $\jdif{\vs}$ continuously extends to $\overline{\Rpos} = [0;+\infty)$ by setting $\jdif{\vs}(0)=0$.

Let us also made the following observation.
Let $\adif,\bdif\in\DRplz$ be any two diffeomorphisms.
Since $\adif(0)=\bdif(0)=0$, we have by Hadamard lemma that $\adif^{-1}(\px) = \px\,\hadif(\px)$ and $\bdif(\px) = \px\,\hbdif(\px)$ for some $\Cr{\kk-1}$ functions $\hadif,\hbdif\colon\bR\to\bR$ such that $\hadif(0) = \adif'(0)>0$ and $\hbdif(0) = \bdif'(0)>0$.
Then on $\overline{\Rpos} = [0;+\infty)$ we have that
\begin{equation}\label{equ:psi_t__b_psi_s__ainv}
\bdif\circ\jdif{\vs}\circ\adif^{-1}(\px)
    =
\bdif(\px^{\vs}\,\hadif(\px)^{\vs})
    =
\px^{\vs}\,
\underbrace{\hadif(\px)^{\vs}\,\hbdif(\px^{\vs}\hadif(\px)^{\vs})}_{\hcdif(\px)}
=
\px^{\vs}\,\hcdif(\px),
\end{equation}
where the function $\hcdif\colon\overline{\Rpos}\to\bR$ is at least continuous.
Now we can prove the lemma.

\ref{enum:lm:x_pow_alpha:EDE}
If $\vs=\vt$, then $\jdif{\vs} = \jdif{\vt}$ and there is nothing to prove.

Conversely, suppose the for $\vs\not=\vt$ the diffeomorphisms $\jdif{\vs}$ and $\jdif{\vt}$ belong to the same $\calE$-double coset of $\dbl{\calE}{\DRpos}{\calE}$, i.e.\ $\jdif{\vt} = \bdif\circ\jdif{\vs}\circ\adif^{-1}$ on $\Rpos$ for some $\adif,\bdif\in\DRplz$ as above.
Not loosing generality, one can assume that $\vs > \vt$.
Then by~\eqref{equ:psi_t__b_psi_s__ainv},
\[
\px^{\vt} = \px^{\vs}\,\hcdif(\px), \quad \px\geq0,
\]
for some continuous function $\hcdif\colon\overline{\Rpos}\to\bR$.
On the other hand, since $\vs>\vt$, the latter identity implies that
\[
    \lim_{\px\to0} \hcdif(\px) = \lim_{\px\to0} \px^{\vt}/\px^{\vs} = +\infty,
\]
which is impossible, as $\hcdif$ is continuous at $0$.
Thus we get a contradiction with the assumption that $\vs\ne\vt$.

Statement~\ref{enum:lm:x_pow_alpha:EDpmE} follows from~\ref{enum:lm:x_pow_alpha:EDE} and the observation that $\jdif{1/\vs} = \jdif{\vs}^{-1}$.
\end{proof}

\begin{subcorollary}\label{cor:nondif_ck_struct_on_Y}
For $\vs>0$ let $\Aatlas_{\vs}$ be any minimal $\Ck$-atlas on $\Ylet$ whose transition map is $\jdif{\vs}$.
Then the following statements hold.
\begin{enumerate}[label={\rm(\arabic*)}]
\item\label{enum:lm:nondif_ck_struct_on_Y:Ys_Yt_non_diff}
If $\vt\not=\vs, 1/\vs$, then $\mpair{\Ylet}{\Aatlas_{\vs}}$ and $\mpair{\Ylet}{\Aatlas_{\vt}}$ are not $\Ck$-diffeomorphic.

\item\label{enum:lm:nondif_ck_struct_on_Y:sne1:Ys_Yt_exch_UV}
For every $\vs\ne1$ the $\Ck$-manifolds $\mpair{\Ylet}{\Aatlas_{\vs}}$ and $\mpair{\Ylet}{\Aatlas_{1/\vs}}$ are  $\Ck$-diffeomorphic, and every $\Ck$-diffeomorphism between them exchanges $\Udomain$ and $\Vdomain$.

\item\label{enum:lm:nondif_ck_struct_on_Y:sne1:Ys_Ys_pres_UV}
For every $\vs\ne1$ every self diffeomorphism of $\mpair{\Ylet}{\Aatlas_{\vs}}$ leaves $\Udomain$ and $\Vdomain$ invariant.

\item\label{enum:lm:nondif_ck_struct_on_Y:s1:canon}
$\mpair{\Ylet}{\Aatlas_{1}}$ is $\Ck$-diffeomorphic to $\mpair{\Ylet}{\CanonAtlas{\Ylet}}$ with the canonical $\Ck$-structure, and it admits a diffeomorphism exchanging $\Udomain$ and $\Vdomain$.
\end{enumerate}
\end{subcorollary}
\begin{proof}
\ref{enum:lm:nondif_ck_struct_on_Y:Ys_Yt_non_diff}
By Lemma~\ref{lm:diffs_of_Y}, $\mpair{\Ylet}{\Aatlas_{\vs}}$ and $\mpair{\Ylet}{\Aatlas_{\vt}}$ are $\Ck$-diffeomorphic iff the corresponding transition maps $\jdif{\vs}$ and $\jdif{\vt}$ belongs the same $(\calE,\pm)$-double coset $\dbli{\calE}{\DRpos}$.
But by Lemma~\ref{lm:x_pow_alpha}\ref{enum:lm:x_pow_alpha:EDE} this is possible iff either $\vt=\vs$ or $\vt=1/\vs$.

\ref{enum:lm:nondif_ck_struct_on_Y:sne1:Ys_Yt_exch_UV}
Let $\vs\ne1$.
Then, as shown in Lemma~\ref{lm:x_pow_alpha}, the identity $\jdif{1/\vs} = (\bdif\circ\jdif{\vs}\circ\adif)^{\delta}$ for some $\adif,\bdif\in\DRplz$ and $\delta\in\{\pm1\}$ is possible only if $\delta=-1$.
Hence by \ref{enum:lm:diffs_of_Y:VU} of Lemma~\ref{lm:diffs_of_Y} every $\Ck$-diffeomorphism between $\mpair{\Ylet}{\Aatlas_{\vs}}$ and $\mpair{\Ylet}{\Aatlas_{1/\vs}}$ exchanges $\Udomain$ and $\Vdomain$.

\ref{enum:lm:nondif_ck_struct_on_Y:sne1:Ys_Ys_pres_UV}
Similarly, again by Lemma~\ref{lm:x_pow_alpha}, for $\vs\ne1$ the identity $\jdif{\vs} = (\bdif\circ\jdif{\vs}\circ\adif)^{\delta}$ for some $\adif,\bdif\in\DRplz$ and $\delta\in\{\pm1\}$ is possible only if $\delta=+1$.
Hence by \ref{enum:lm:diffs_of_Y:UV} of Lemma~\ref{lm:diffs_of_Y} every $\Ck$-diffeomorphism of $\mpair{\Ylet}{\Aatlas_{\vs}}$ leaves invariant $\Udomain$ and $\Vdomain$.

\ref{enum:lm:nondif_ck_struct_on_Y:s1:canon}
This statement follows from the observation that $\jdif{1} = \id_{\Rpos}$ and Lemma~\ref{lm:x_pow_alpha}.
\end{proof}

\begin{subremark}\rm
The proof of Theorem~\ref{lm:diffs_of_Y} literally repeats the arguments of the classification of $\Ck$-structures on the line with two origins $\DLine$.
One easily sees that the arguments consist of checking a certain commutative diagram.
In next section we will apply them in an essentially more general context.
\end{subremark}

\section{Atlases with two charts}
\label{sect:atlases_with_two_charts}
\newcommand\CkAUV{\mathfrak{C}^{\kk}(\Aspace,\Uatlas,\Vatlas)}

In this section we will reformulate certain properties of $\Ylet$ under more general settings.

\subsection{$\UVAtlas$-atlases}
Let $\mpair{\Aspace}{\Latlas}$ be a $\Ck$-manifold, and $\Udomain,\Vdomain\subset\Aspace$ be two open subsets such that $\Aspace = \Udomain\cup\Vdomain$.
Put $\Wdomain = \Udomain\cap\Vdomain$.
Consider the restrictions of the atlas $\Latlas$ to those open sets:
\begin{align*}
    \Uatlas &= \AtlasRestr{\Latlas}{\Udomain}, &
    \Vatlas &= \AtlasRestr{\Latlas}{\Vdomain}, &
    \Watlas &= \AtlasRestr{\Latlas}{\Wdomain}.
\end{align*}
Then $\Uatlas$ and $\Vatlas$ are $\Ck$-compatible partial atlases on $\Aspace$ such that their union $\Uatlas\cup\Vatlas$ is a $\Ck$-atlas on $\Aspace$ being $\Ck$-compatible with $\Latlas$.

We are interested in the following problem: \term{to what extent the \term{$\Ck$-manifolds} $\mpair{\Udomain}{\Uatlas}$ and $\mpair{\Vdomain}{\Vatlas}$ determine $\Ck$-structures on $\Aspace$?}

More precisely, suppose we are given two homeomorphisms
$\Ucmap\colon\Udomain\to\Udomain$ and $\Vcmap\colon\Vdomain\to\Vdomain$ such that $\Ucmap(\Wdomain)=\Wdomain=\Vcmap(\Wdomain)$.
They can be regarded as $\Ck$-diffeomorphisms $\Ucmap\colon\mpair{\Udomain}{\indatl{(\Ucmap^{-1})}{\Uatlas}}\to\mpair{\Udomain}{\Uatlas}$ and $\Vcmap\colon\mpair{\Vdomain}{\indatl{(\Vcmap^{-1})}{\Vatlas}}\to\mpair{\Vdomain}{\Vatlas}$ with respect to the induced ``pull back'' atlases $\indatl{(\Ucmap^{-1})}{\Uatlas}$ and $\indatl{(\Vcmap^{-1})}{\Vatlas}$.
Suppose also that those atlases are $\Ck$-compatible.
This is equivalent to the assumption that the ``transition map'' $\gdif=\Vcmap\circ\Ucmap^{-1}\colon\Wdomain\to\Wdomain$ is a $\Ck$-diffeomorphism of $\mpair{\Wdomain}{\Watlas}$.
Then $\uva{} = \indatl{(\Ucmap^{-1})}{\Uatlas} \cup \indatl{(\Vcmap^{-1})}{\Vatlas}$ is a $\Ck$-atlas on $\Aspace$.

One can think that a $\Ck$-manifold $\mpair{\Aspace}{\Latlas}$ is \term{glued} from $\Ck$-manifolds $\mpair{\Udomain}{\Uatlas}$ and $\mpair{\Vdomain}{\Vatlas}$ by the identity $\Ck$-diffeomorphism $\id_{\Wdomain}$, while $\mpair{\Aspace}{\uva{}}$ is glued from $\mpair{\Udomain}{\Uatlas}$ and $\mpair{\Vdomain}{\Vatlas}$ via the $\Ck$-diffeomorphism $\gdif$.

Also, in the spirit of~\cite[Definition~6.9]{Munkres:AnnMath:1968}, one can also say that $\mpair{\Aspace}{\uva{}}$ is obtained by \term{pasting $\Udomain$ and $\Vdomain$ into $\mpair{\Aspace}{\Latlas}$ via $\Ucmap$ and $\Vcmap$}.

Note that a priori $\mpair{\Aspace}{\Latlas}$ and $\mpair{\Aspace}{\uva{}}$ might not be $\Ck$-diffeomorphic.
In this section we will classify $\Ck$-structures $\mpair{\Aspace}{\uva{}}$ up to a $\Ck$-diffeomorphism leaving invariant $\Udomain$ and $\Vdomain$ or exchanging them, see Theorems~\ref{th:CkLUV_HLUV} and~\ref{th:act_DUWZ2_DW} below.

\begin{subdefinition}
An \term{$\UVAtlas$-atlas on $\Aspace$} is a pair $\{ \Ucmap\colon\Udomain\to\Udomain, \Vcmap\colon\Vdomain\to\Vdomain\}$ of \term{homeomorphisms} such that $\Ucmap(\Wdomain) = \Wdomain = \Vcmap(\Wdomain)$ and the induced ``\term{transition map}''
\[
    \erestr{(\Vcmap\circ\Ucmap^{-1})}{\Wdomain}\colon\mpair{\Wdomain}{\Watlas}\to\mpair{\Wdomain}{\Watlas}
\]
is a \term{$\Ck$-diffeomorphism}.
Such an atlas will also be denoted by $\{(\Udomain,\Ucmap), (\Vdomain,\Vcmap)\}$.
\end{subdefinition}
Thus an $\UVAtlas$-atlas can be regarded as a $\Ck$-atlas whose coordinate maps take values in the disjoint union $\mpair{\Udomain}{\Uatlas} \sqcup \mpair{\Vdomain}{\Vatlas}$ instead of some cones.

The following lemma is trivial and just repeats what was said before on $\UVAtlas$-atlases.
We leave it for the reader.
\begin{sublemma}\label{lm:existence_of_UV_atlases}
{\rm(1)}~For each $\UVAtlas$-atlas $\Aatlas=\{(\Udomain,\Ucmap), (\Vdomain,\Vcmap)\}$ the collection of charts:
\[
    \uva{\Aatlas} \,=\,
        \indatl{(\Ucmap^{-1})}{\Uatlas}
        \,\cup\,
        \indatl{(\Vcmap^{-1})}{\Vatlas}
\]
constitute a $\Ck$-atlas on $\Aspace$ such that the following commutative diagram in the arrows are natural inclusions being also \term{open $\Ck$-embeddings}:
\begin{equation*}
\begin{gathered}
    \xymatrix{
        \mpair{\Wdomain}{\Watlas} \ar[r]^-{\wvinc} \ar[d]_-{\wuinc} &
        \mpair{\Vdomain}{\Vatlas} \ar[d]^{\Vcmap^{-1}} \\
        \mpair{\Udomain}{\Uatlas} \ar[r]^-{\Ucmap^{-1}} &
        \mpair{\Aspace}{\uva{\Aatlas}}
    }
\end{gathered}
\end{equation*}

{\rm(2)}~Conversely, let $\Batlas$ be a $\Ck$-structure on $\Aspace$ such that there exist $\Ck$-diffeomorphisms
\begin{align*}
    &\Ucmap\colon\mpair{\Udomain}{\AtlasRestr{\Batlas}{\Udomain}}\to\mpair{\Udomain}{\Uatlas}, &
    &\Vcmap\colon\mpair{\Vdomain}{\AtlasRestr{\Batlas}{\Vdomain}}\to\mpair{\Vdomain}{\Vatlas}.
\end{align*}
Then the pair $\Aatlas=\{(\Udomain,\Ucmap), (\Vdomain,\Vcmap)\}$ is an $\UVAtlas$-atlas, such that the corresponding $\Ck$-structure $\uva{\Aatlas}$ is $\Ck$-compatible with $\Batlas$, i.e.\ $\uva{\Aatlas}\subset\Batlas$.
\end{sublemma}

Given an $\UVAtlas$-atlas $\Aatlas=\{(\Udomain,\Ucmap), (\Vdomain,\Vcmap)\}$ on $\Aspace$, denote by $\mpair{\Aspace}{\Aatlas}$, instead of $\mpair{\Aspace}{\uva{\Aatlas}}$, the manifold $\Aspace$ endowed with the corresponding $\Ck$-structure $\uva{\Aatlas}$ determined by $\Aatlas$.

In particular, the pair of identity maps $\{ (\Udomain, \id_{\Udomain}), (\Vdomain, \id_{\Vdomain})\}$ is an
$\UVAtlas$-atlas inducing the initial $\Ck$-structure $\Latlas$ on $\Aspace$, and therefore we will denote this atlas by the same letter:
\begin{equation}\label{equ:atlas_C__U_V}
    \Latlas = \{ (\Udomain, \id_{\Udomain}), (\Vdomain, \id_{\Vdomain})\}.
\end{equation}
This will never lead to confusion.

Denote by $\CkAUV$ the set of all $\Ck$-structures on $\Aspace$ containing a $\Ck$-atlas $\uva{\Aatlas}$ for some $\UVAtlas$-atlas $\Aatlas$.

Further, let $\HomeoLUV$ be the group of \term{homeomorphisms of the triple} $(\Aspace,\Udomain,\Vdomain)$, i.e.\ homeomorphisms $\dif\colon\Aspace\to\Aspace$ such that $\dif(\Udomain)=\Udomain$ and $\dif(\Vdomain)=\Vdomain$.
In particular, $\dif(\Wdomain)=\Wdomain$.

Let also $\HomeoLPartUV$ be the group of homeomorphisms of $\Aspace$ which leave $\Udomain$ and $\Vdomain$ invariant or exchange them.
\begin{subremark}\rm
Note that a priori there might not exist a homeomorphism $\uvdif\colon\Aspace\to\Aspace$ exchanging $\Udomain$ and $\Vdomain$, e.g.\ when $\Udomain$ and $\Vdomain$ are not homeomorphic, and in that case $\HomeoLUV=\HomeoLPartUV$.
However, if such $\uvdif$ exists, then $\HomeoLUV$ is a normal subgroup of $\HomeoLPartUV$ of index $2$.

On the other hand, if for some $\Batlas,\Batlas'\in\CkAUV$ there exists a $\Ck$-diffeomorphism $\uvdif\colon\mpair{\Aspace}{\Batlas}\to\mpair{\Aspace}{\Batlas'}$ exchanging $\Udomain$ and $\Vdomain$, i.e.\ $\uvdif(\Udomain)=\Vdomain$ and $\uvdif(\Vdomain)=\Udomain$, then $\mpair{\Udomain}{\Uatlas}$ and $\mpair{\Vdomain}{\Vatlas}$ are $\Ck$-diffeomorphic.
\end{subremark}

The following lemma is easy and straightforward, and we leave its verification for the reader.
\begin{sublemma}
Let $\Aatlas=\{(\Udomain,\Ucmap), (\Vdomain,\Vcmap)\}$ be an $\UVAtlas$-atlas on $\Aspace$ and
\[
    \uva{\Aatlas}
        \,=\,
    \indatl{(\Ucmap^{-1})}{\Uatlas}
        \ \cup\
    \indatl{(\Vcmap^{-1})}{\Vatlas}
\]
be the induced $\Ck$-atlas.

{\rm~1)}~Then for every $\dif\in\HomeoLUV$ the pair of homeomorphisms
\begin{equation}\label{equ:h_A_UU_VV}
    \indatl{\dif}{\Aatlas}:=
    \bigl\{
        \Vcmap\circ\dif^{-1}\colon\Udomain\to\Udomain, \
        \Ucmap\circ\dif^{-1}\colon\Vdomain\to\Vdomain
    \bigr\}
\end{equation}
is an $\UVAtlas$-atlas as well and
\[
    \indatl{\dif}{\uva{\Aatlas}} = \uva{\indatl{\dif}{\Aatlas}}.
\]
Moreover, if $\kdif\in\HomeoLUV$ is another homeomorphism, then $\indatl{(\kdif\circ\dif)}{\Aatlas}=\indatl{\kdif}{\indatl{\dif}{\Aatlas}}$.

Thus we get a well-defined left action of $\HomeoLUV$ on the set of all $\UVAtlas$-atlases, which agrees with the usual action of $\HomeoLUV$ on the set of all $\Ck$-structures on $\Aspace$.
In particular, $\CkAUV$ is invariant under the action of $\HomeoLUV$.

{\rm~2)}~Suppose that there exists a $\Ck$-diffeomorphism $\uvdif\colon\mpair{\Udomain}{\Uatlas}\to\mpair{\Vdomain}{\Vatlas}$ such that $\uvdif(\Wdomain)=\Wdomain$.
Then for each $\dif\in\HomeoLPartUV$ exchanging $\Udomain$ and $\Vdomain$, i.e.\ belonging to the complement $\HomeoLPartUV\setminus\HomeoLUV$, the following pair of homeomorphism
\begin{equation}\label{equ:h_A_UV_VU}
    \indatl{\dif}{\Aatlas}:=
    \bigl\{
        \uvdif\circ\Vcmap\circ\dif^{-1}\colon\Udomain\to\Udomain, \
        \uvdif^{-1}\circ\Ucmap\circ\dif^{-1}\colon\Vdomain\to\Vdomain
    \bigr\}
\end{equation}
is still an $\UVAtlas$-atlas.
Moreover, again, for any $\dif,\kdif\in\HomeoLPartUV$, we have that
\begin{align*}
    &\indatl{\dif}{\uva{\Aatlas}} = \uva{\indatl{\dif}{\Aatlas}}, &
    &\indatl{(\kdif\circ\dif)}{\Aatlas}=\indatl{\kdif}{\indatl{\dif}{\Aatlas}},
\end{align*}
where $\indatl{\dif}{\cdot}$ is defined by the respective formula~\eqref{equ:h_A_UU_VV} or~\eqref{equ:h_A_UV_VU}.

Hence, we again obtain a well-defined (but depending on a fixed $\uvdif$) action of $\HomeoLPartUV$ on the set of all $\UVAtlas$-atlases, which agrees with the usual action of $\HomeoLUV$ on the set of all $\Ck$-structures on $\Aspace$.
In particular, $\CkAUV$ is still invariant under the action of $\HomeoLUV$.
\end{sublemma}

Thus our aim is to describe the orbit spaces
\[
\CkAUV\,/\,\HomeoLUV
\qquad \text{and} \qquad
\CkAUV\,/\,\HomeoLPartUV.
\]

\subsection{Diffeomorphisms preserving $\Udomain$ and $\Vdomain$}
Let $\DiffUW$ and $\DiffVW$ be respectively the groups of $\Ck$-diffeomorphisms of $\mpair{\Udomain}{\Uatlas}$ and $\mpair{\Vdomain}{\Vatlas}$ which leave $\Wdomain$ invariant, and $\DiffW$ be the group of $\Ck$-diffeomorphisms of $\mpair{\Wdomain}{\Watlas}$.
Then we have the restriction homomorphisms:
\begin{align*}
    & \rho_{\Udomain,\Wdomain}\colon\DiffUW\to\DiffW, & & \ \rho_{\Udomain,\Wdomain}(\dif) = \erestr{\dif}{\Wdomain}, \\
    & \rho_{\Vdomain,\Wdomain}\colon\DiffVW\to\DiffW, & & \ \rho_{\Vdomain,\Wdomain}(\dif) = \erestr{\dif}{\Wdomain}.
\end{align*}
Denote by
\begin{align}\label{equ:EdiffWU}
    &\EDiffWU:=\rho_{\Udomain,\Wdomain}(\DiffUW),&
    &\EDiffWV:=\rho_{\Vdomain,\Wdomain}(\DiffVW)
\end{align}
the respective images of these homomorphisms.
They consist of $\Ck$-diffeomorphism of $\mpair{\Wdomain}{\Watlas}$ which can be extended to $\Ck$-diffeomorphisms of $\Udomain$ and $\Vdomain$ respectively.
In particular, one can define the following set of \term{$(\EDiffWV,\EDiffWU)$-double cosets}:
\[
    \dbl{\EDiffWV}{\DiffW}{\EDiffWU},
\]
consisting of the orbits of the action of $\EDiffWV \times \EDiffWU$ on $\DiffW$ given by $(\bdif,\adif)\cdot\gdif = \bdif\circ\gdif\circ\adif^{-1}$.

The following theorem can be proved by almost literally the same arguments as~\ref{enum:lm:diffs_of_Y:UV} of Lemma~\ref{lm:diffs_of_Y} and the first part of Theorem~\ref{th:Ck_str_on_Y}.
\begin{subtheorem}\label{th:CkLUV_HLUV}
The correspondence
\begin{equation*}
\Aatlas=\{(\Udomain,\Ucmap),(\Vdomain,\Vcmap)\}
\ \longmapsto \
\Atrmap = \Vcmap\circ\Ucmap^{-1}
\end{equation*}
associating to each $\UVAtlas$-atlas $\Aatlas$ on $\Aspace$ its transition map $\Atrmap$ yields an \term{injective map} $\mu'$
\begin{itemize}
\item
from the set $\CkAUV\,/\,\HomeoLUV$ of all $\Ck$-structures on $\Aspace$ induced by some $\UVAtlas$-atlas up to a $\Ck$-diffeomorphism leaving invariant $\Udomain$ and $\Vdomain$
\item
to the set \ $\dbl{\EDiffWV}{\DiffW}{\EDiffWU}$ \ of $(\EDiffWV,\EDiffWU)$-double cosets.
\end{itemize}
Moreover, $\mu'$ is a bijection if and only if every $\gdif\in\DiffW$ is a transition map for some $\UVAtlas$-atlas $\Aatlas$.
\end{subtheorem}
\begin{proof}
For an $\UVAtlas$-atlas $\Aatlas$ denote by $[\Aatlas]'$ its equivalence class in
\[ \CkAUV\,/\,\HomeoLUV,\]
and let $[\Atrmap]'$ be the $(\EDiffWV,\EDiffWU)$-double coset in $\dbl{\EDiffWV}{\DiffW}{\EDiffWU}$ of its transition map.

Let $\Aatlas = \{(\Udomain,\aUcmap), (\Vdomain,\aVcmap)\}$ and $\Batlas = \{(\Udomain,\bUcmap), (\Vdomain,\bVcmap)\}$ be two $\UVAtlas$-atlases on $\Aspace$.
We claim that similarly to~\ref{enum:lm:diffs_of_Y:UV} of Lemma~\ref{lm:diffs_of_Y} the following conditions are equivalent:
\begin{enumerate}[label={\rm(\alph*)}]
\item\label{enum:CkLUV__to__HLUV}
there exists a $\Ck$-diffeomorphism $\dif\colon\mpair{\Aspace}{\Aatlas}\to\mpair{\Aspace}{\Batlas}$ belonging to $\HomeoLUV$;
\item\label{enum:HLUV__to__CkLUV}
the corresponding transition maps $\Atrmap=\aVcmap\circ\aUcmap^{-1}$ and $\Btrmap=\bVcmap\circ\bUcmap^{-1}$ belong to the same $(\EDiffWV,\EDiffWU)$-double coset.
\end{enumerate}
Then, as in the proof of the first part of Theorem~\ref{th:Ck_str_on_Y}, the implication~\ref{enum:CkLUV__to__HLUV}$\Rightarrow$\ref{enum:HLUV__to__CkLUV} will guarantee that the correspondence
\[
\CkAUV\,/\,\HomeoLUV \ \ni \ [\Aatlas]' \ \xrightarrow{~~\mu'~~} \ [\Atrmap]' \ \in \ \dbl{\EDiffWV}{\DiffW}{\EDiffWU}
\]
is well defined, while the inverse implication~\ref{enum:HLUV__to__CkLUV}$\Rightarrow$\ref{enum:CkLUV__to__HLUV} will show that $\mu'$ is injective.

Before proving the equivalence \ref{enum:CkLUV__to__HLUV}$\Leftrightarrow$\ref{enum:HLUV__to__CkLUV}, note that for each $\dif\in\HomeoLUV$ we have the following commutative diagram similar to~\eqref{equ:h_UU_VV__Y}:
\begin{equation}\label{equ:h_UU_VV__L_general}
\begin{gathered}
\xymatrix@C=4.5em@R=1.5em{
    \Udomain
        \ar@{-->}@/_6ex/[ddd]_-{\adif}
    &
    \Wdomain
        \ar@{_(->}[l]_-{\wuinc}
        \ar[r]^-{\Atrmap = \aVcmap \circ \aUcmap^{-1}}
        \ar@{-->}@/^2.5ex/[ddd]^-{\erestr{\adif}{\Wdomain}}
    &
    \Wdomain
        \ar@{^(->}[r]^-{\wvinc}
        \ar@{-->}@/_2.5ex/[ddd]_-{\erestr{\bdif}{\Wdomain}}
    &
    \Vdomain
        \ar@{-->}@/^6ex/[ddd]^-{\bdif}
    \\
    \Udomain
        \ar[u]_-{\aUcmap}
        \ar[d]_-{\erestr{\dif}{\Udomain}}
    &
    \Wdomain
        \ar@{_(->}[l]_-{\wuinc}
        \ar@{=}[r]
        \ar[u]^-{\erestr{\aUcmap}{\Wdomain}}
        \ar[d]_-{\erestr{\dif}{\Wdomain}}
    &
    \Wdomain
        \ar@{^(->}[r]^-{\wvinc}
        \ar[u]_-{\erestr{\aVcmap}{\Wdomain}}
        \ar[d]^-{\erestr{\dif}{\Wdomain}}
    &
    \Vdomain
        \ar[u]^-{\aVcmap}
        \ar[d]^-{\erestr{\dif}{\Vdomain}}
    \\
    \Udomain
        \ar[d]^-{\bUcmap}
    &
    \Wdomain
        \ar@{_(->}[l]_-{\wuinc}
        \ar@{=}[r]
        \ar[d]_-{\erestr{\bUcmap}{\Wdomain}}
    &
    \Wdomain
        \ar@{^(->}[r]^-{\wvinc}
        \ar[d]^-{\erestr{\bVcmap}{\Wdomain}}
    &
    \Vdomain
        \ar[d]_-{\bVcmap}
    \\
    \Udomain
    &
    \Wdomain
        \ar@{_(->}[l]_-{\wuinc}
        \ar[r]^-{\Btrmap = \bVcmap\circ\bUcmap^{-1}}
    &
    \Wdomain
        \ar@{^(->}[r]^-{\wvinc}
    &
    \Vdomain
}
\end{gathered}
\end{equation}
in which commutativity of corner squares means that $\Wdomain$ is invariant under $\aUcmap$, $\aVcmap$, $\bUcmap$, and $\bVcmap$, see Section~\ref{sect:restriction_map}.
This implies that the homeomorphisms $\adif\colon\Udomain\to\Udomain$ and $\bdif\colon\Vdomain\to\Vdomain$ preserve $\Wdomain$ as well.
They can regarded as \term{representation of $\dif$ with respect to the atlas $\Aatlas$}.
Denote $\hadif:=\erestr{\adif}{\Wdomain}$ and $\hbdif:=\erestr{\bdif}{\Wdomain}$.
By commutativity of the middle vertical rectangle in~\eqref{equ:h_UU_VV__L_general} we have the following identity:
\begin{equation}\label{equ:L_gb_b_ga_ainv}
    \Btrmap = \hbdif \circ \Atrmap \circ \hadif^{-1}.
\end{equation}

\ref{enum:CkLUV__to__HLUV}$\Rightarrow$\ref{enum:HLUV__to__CkLUV}
Notice that $\dif$ is $\Ck$ iff $\adif$ and $\bdif$ are $\Ck$-diffeomorphisms of $\mpair{\Udomain}{\Uatlas}$ and $\mpair{\Vdomain}{\Vatlas}$ respectively, i.e.\ $\adif\in\DiffUW$ and $\bdif\in\DiffVW$, and in that case~\eqref{equ:L_gb_b_ga_ainv} shows that $\Atrmap$ and $\Btrmap$ belong to the same $(\EDiffWV,\EDiffWU)$-double coset.

\ref{enum:HLUV__to__CkLUV}$\Rightarrow$\ref{enum:CkLUV__to__HLUV}
Suppose $\Btrmap = \erestr{\bdif}{\Wdomain} \circ \Atrmap \circ \erestr{\adif^{-1}}{\Wdomain}$ for some $\adif\in\DiffUW$ and $\bdif\in\DiffVW$.
Then one easily checks that the map $\dif\colon\Aspace\to\Aspace$ given by~\eqref{equ:h_Y_Y__reconstr_via_a_b} is a $\Ck$-diffeomorphism $\mpair{\Aspace}{\Aatlas}\to\mpair{\Aspace}{\Batlas}$.
\end{proof}

\subsection{Diffeomorphisms exchanging $\Udomain$ and $\Vdomain$}
Suppose that there exists a $\Ck$-diffeo\-morphism $\uvdif\colon\mpair{\Udomain}{\Uatlas}\to\mpair{\Vdomain}{\Vatlas}$ such that $\uvdif(\Wdomain)=\Wdomain$.

In this subsection we will describe the orbit space $\CkAUV\,/\,\HomeoLPartUV$, i.e.\ classify $\Ck$-structures on $\Aspace$ belonging to $\CkAUV$ up to a $\Ck$-diffeomorphism which either leaves $\Udomain$ and $\Vdomain$ invariant or exchange them.

Denote by $\DiffUVW$ the set of $\Cr{\kk}$-diffeomorphisms $\dif\colon\mpair{\Udomain}{\Uatlas}\to\mpair{\Vdomain}{\Vatlas}$ such that $\dif(\Wdomain)=\Wdomain$.
Similarly, let $\DiffVUW$ be the set $\Cr{\kk}$-diffeomorphisms $\dif\colon\mpair{\Vdomain}{\Vatlas}\to\mpair{\Udomain}{\Uatlas}$ such that $\dif(\Wdomain)=\Wdomain$.
Notice that they are \term{sets} and have no canonical group structure.
Also, there is a canonical bijection: $\DiffUVW \ni \eta \mapsto \eta^{-1} \in \DiffVUW$.
Moreover, every $\eta\in\DiffVUW$ yields the following three bijections:
\begin{equation}\label{equ:bij_DiffUV_DiffVU}
\begin{gathered}
\xymatrix{
    & \DiffUVW \ar[dr]^-{\ \xi\, \mapsto\, \xi\, \circ\, \eta^{-1}} \\
    \DiffUW \ar[ru]^-{\dif\, \mapsto\, \eta\, \circ\, \dif}
            \ar[rr]^-{\dif\, \mapsto\, \eta\, \circ\, \dif\, \circ\, \eta^{-1}} &&
    \DiffVW
}
\end{gathered}
\end{equation}
such that the lower horizontal arrow is also an isomorphism of groups.

\begin{sublemma}\label{lm:h_UV_VU}
Let $\Aatlas = \{(\Udomain,\aUcmap), (\Vdomain,\aVcmap)\}$ and $\Batlas = \{(\Udomain,\bUcmap), (\Vdomain,\bVcmap)\}$ be two $\UVAtlas$-atlases on $\Aspace$.
Then the following conditions are equivalent:
\begin{enumerate}[label={\rm(\alph*)}]
\item\label{enum:th:h_UV_VU:diff}
there exists a $\Cr{\kk}$-diffeomorphism $\dif\colon\mpair{\Aspace}{\Aatlas}\to\mpair{\Aspace}{\Batlas}$ exchanging $\Udomain$ and $\Vdomain$;
\item\label{enum:th:h_UV_VU:dbl_cosest}
there exist $\adif\in\DiffUVW$ and $\bdif\in\DiffVUW$ such that
\[ \Btrmap = ( \erestr{\bdif}{\Wdomain} \circ \Atrmap \circ \erestr{\adif}{\Wdomain}^{-1})^{-1}.\]
\end{enumerate}
\end{sublemma}
\begin{proof}
Similarly to Theorem~\ref{th:CkLUV_HLUV}, the proof of our lemma almost literally repeats the arguments of statement~\ref{enum:lm:diffs_of_Y:VU} of Lemma~\ref{lm:diffs_of_Y}.
Let us just note that every $\dif\colon\mpair{\Aspace}{\Aatlas}\to\mpair{\Aspace}{\Batlas}$ exchanging $\Udomain$ and $\Vdomain$ yields the following commutative diagram which differs from~\eqref{equ:h_UU_VV__L_general} by reversing the last two rows:
\begin{equation}\label{equ:h_UV_VU__L_general}
\begin{gathered}
\xymatrix@C=4.5em@R=1.5em{
    \Udomain
        \ar@{-->}@/_6ex/[ddd]_-{\adif}
    &
    \Wdomain
        \ar@{_(->}[l]_-{\wuinc}
        \ar[r]^-{\Atrmap = \aVcmap \circ \aUcmap^{-1}}
        \ar@{-->}@/^2.5ex/[ddd]^-{\erestr{\adif}{\Wdomain}}
    &
    \Wdomain
        \ar@{^(->}[r]^-{\wvinc}
        \ar@{-->}@/_2.5ex/[ddd]_-{\erestr{\bdif}{\Wdomain}}
    &
    \Vdomain
        \ar@{-->}@/^6ex/[ddd]^-{\bdif}
    \\
    \Udomain
        \ar[u]_-{\aUcmap}
        \ar[d]_-{\erestr{\dif}{\Udomain}}
    &
    \Wdomain
        \ar@{_(->}[l]_-{\wuinc}
        \ar@{=}[r]
        \ar[u]^-{\erestr{\aUcmap}{\Wdomain}}
        \ar[d]_-{\erestr{\dif}{\Wdomain}}
    &
    \Wdomain
        \ar@{^(->}[r]^-{\wvinc}
        \ar[u]_-{\erestr{\aVcmap}{\Wdomain}}
        \ar[d]^-{\erestr{\dif}{\Wdomain}}
    &
    \Vdomain
        \ar[u]^-{\aVcmap}
        \ar[d]^-{\erestr{\dif}{\Vdomain}}
    \\
    \Vdomain
        \ar[d]^-{\bVcmap}
    &
    \Wdomain
        \ar@{_(->}[l]_-{\wvinc}
        \ar@{=}[r]
        \ar[d]_-{\erestr{\bVcmap}{\Wdomain}}
    &
    \Wdomain
        \ar@{^(->}[r]^-{\wuinc}
        \ar[d]^-{\erestr{\bUcmap}{\Wdomain}}
    &
    \Udomain
        \ar[d]_-{\bUcmap}
    \\
    \Vdomain
    &
    \Wdomain
        \ar@{_(->}[l]_-{\wvinc}
    &
    \Wdomain
        \ar@{^(->}[r]^-{\wuinc}
        \ar[l]_-{\, \Btrmap = \bVcmap\circ\bUcmap^{-1}}
    &
    \Udomain
}
\end{gathered}
\end{equation}
We leave the details for the reader.
\end{proof}

Now the bijections~\eqref{equ:bij_DiffUV_DiffVU} allow to rephrase Lemma~\ref{lm:h_UV_VU} in terms of the group $\DiffUW$ of diffeomorphisms of $\Udomain$ leaving $\Wdomain$ invariant.
Namely, recall that we denoted by $\EDiffWU$ the subgroup of $\DiffW$ consisting of $\Ck$-diffeomorphisms which can be extended to $\Ck$-diffeomorphisms of all on $\Udomain$, see~\eqref{equ:EdiffWU}.

Our next step is to define a certain action of $\EDiffWU\wr\bZ_{2}$ on $\DiffW$ and construct a bijection between $\CkAUV\,/\,\HomeoLPartUV$ and the orbits set $\DiffW/\bigl(\EDiffWU\wr\bZ_{2}\bigr)$, similarly to the second part of Theorem~\ref{th:Ck_str_on_Y}.
First we make the following observation motivating the definition of that action.

\begin{subremark}\rm
Let $\dif\in\HomeoLPartUV$, and $\adif$ and $\bdif$ be the coordinate representation of $\restr{\dif}{\Udomain}$ and $\restr{\dif}{\Vdomain}$ with respect to the atlases $\Uatlas$ and $\Vatlas$.
Then, depending on whether $\dif$ preserves $\Udomain$ and $\Vdomain$ or exchanges them, $\adif$ is either a map $\Udomain\to\Udomain$ or a map $\Udomain\to\Vdomain$, and similarly for $\bdif$.
In particular, the domain of $\adif$ is always $\Udomain$, while the domain of $\bdif$ is $\Vdomain$.

In turn, let $\hadif=\erestr{\adif}{\Wdomain}$ and $\hbdif=\erestr{\bdif}{\Wdomain}$ be the restrictions of $\adif$ and $\bdif$ to $\Wdomain$ respectively.
Since $\adif$ and $\bdif$ preserve $\Wdomain$, we have that $\hadif,\hbdif\in\DiffW$, and thus they ``forget'' the domain and target spaces of $\adif$ and $\bdif$.
We can ``restore'' that information as follows.

First, introduce a number $\delta=+1$ if $\dif$ preserves $\Udomain$ and $\Vdomain$, and $\delta=-1$ whenever $\dif$ exchanges them.
Now if $\delta=+1$, $\hadif$ is a self-diffeomorphism of $\Wdomain$ regarded as a subset of $\Udomain$ and thus belongs to $\EDiffWU$.
On the other hand, $\hbdif$ is a self-diffeomorphism of $\Wdomain$ regarded as a subset of $\Vdomain$.
Therefore, in order to make $\hbdif$ a ``proper'' element of $\EDiffWU$, we should replace it with $\hbdif'=\uvdif^{-1}\circ\hbdif\circ\uvdif$.

In the case $\delta=-1$, $\hadif$ is a self-diffeomorphism of $\Wdomain$ so that the domain of $\hadif$ is contained in $\Udomain$, while the target is in $\Vdomain$.
Therefore,  in order to make it an element of $\EDiffWU$, we should replace $\hadif$ with $\hadif'=\uvdif^{-1}\circ\hadif$.
Similarly, $\hbdif$ should be replaced with $\hbdif'=\hbdif\circ\uvdif$.

Analogously, the transition maps $\Atrmap$ and $\Btrmap$ are diffeomorphisms from the subset $\Wdomain \subset\Udomain$ onto the subset $\Wdomain\subset\Vdomain$.
Hence, in order to regard them self-diffeomorphisms of $\Wdomain$ as a subset of $\Udomain$, one should replace them with $\eta^{-1}\Atrmap$ and $\eta^{-1}\Btrmap$ respectively.

Finally, we obtain the following commutative diagrams and the identities showing that the triples $(\hadif', \hbdif',\delta)$ ``act'' on ``adapted'' transition maps $\eta^{-1}\Atrmap$ and $\eta^{-1}\Btrmap$ similarly to the action from Lemma~\ref{lm:diffs_of_Y}:
\begin{equation}\label{equ:motivation_for_actions}
\begin{array}{ccc}
\xymatrix{
    \Udomain & \Vdomain & \Udomain \\
    \Wdomain \ar@{^(->}[u]^-{\wuinc}
             \ar[r]^-{\Atrmap}
             \ar[d]_-{\hadif' = \hadif}
             \ar@{}[dr]|{\bullet} &
    \Wdomain \ar@{^(->}[u]^-{\wvinc} \ar[d]^-{\hbdif} &
    \Wdomain \ar@{^(->}[u]^-{\wuinc}
             \ar[l]_-{\uvdif}
             \ar[d]^-{\hbdif'}
    \\
    \Wdomain \ar@{_(->}[d]_-{\wuinc}
             \ar[r]^-{\Btrmap}
             &
    \Wdomain \ar@{_(->}[d]_-{\wvinc} &
    \Wdomain \ar@{_(->}[d]_-{\wuinc}
             \ar[l]_-{\uvdif}
    \\
    \Udomain & \Vdomain & \Udomain
}
&\qquad &
\xymatrix{
    & \Udomain & \Vdomain & \Udomain
    \\
    & \Wdomain \ar@{^(->}[u]^-{\wuinc}
             \ar[r]^-{\Atrmap}
             \ar[d]_-{\hadif} \ar[dl]_-{\hadif'}
             \ar@{}[dr]|{\bullet}
             &
    \Wdomain \ar@{^(->}[u]^-{\wvinc} \ar[d]^-{\hbdif}  &
    \Wdomain \ar@{^(->}[u]^-{\wuinc}
             \ar[l]_-{\uvdif}
             \ar[dl]^-{\hbdif'}
    \\
    \Wdomain \ar@{_(->}[d]_-{\wuinc} \ar[r]^-{\uvdif} &
    \Wdomain \ar@{_(->}[d]_-{\wvinc}
              &
    \Wdomain \ar@{_(->}[d]_-{\wuinc} \ar[l]_-{\Btrmap}
    \\
    \Udomain & \Vdomain & \Udomain &
} \\
\uvdif^{-1} \circ \Btrmap            = \hbdif' \circ (\uvdif^{-1}\circ\Atrmap) \circ (\hadif')^{-1} &&
(\uvdif^{-1}\circ \Btrmap^{-1})^{-1} = \hbdif' \circ (\uvdif^{-1}\circ\Atrmap) \circ (\hadif')^{-1}
\end{array}
\end{equation}
Notice that the latter identities are equivalent to the commutativity of the corresponding squares marked by the symbol ``$\bullet$''.
\end{subremark}

The following theorem classifies the orbits space $\CkAUV\,/\,\HomeoLPartUV$.
To simplify notations, in the following theorem we omit the symbols of composition ``$\circ$'' and restriction ``$\erestr{}{\Wdomain}$'' to $\Wdomain$, since all mentioned maps leave $\Wdomain$ invariant.

\begin{subtheorem}\label{th:act_DUWZ2_DW}
{\rm 1)}~For every $\eta\in\DiffUVW$ there is an action of $\EDiffWU\wr\bZ_{2}$ on $\DiffW$ given by the following rule: if $(\adif,\bdif,\delta)\in\EDiffWU\wr\bZ_{2}$ and $\gdif\in\DiffW$, then
\begin{equation*}
(\adif,\bdif,\delta)\cdot\gdif =
\begin{cases}
\uvdif (\bdif\uvdif^{-1} \gdif \adif^{-1}) =
(\uvdif\bdif\uvdif^{-1}) \gdif \adif^{-1},  & \delta = +1,
\\
\uvdif (\bdif\uvdif^{-1} \gdif \adif^{-1})^{-1}  =
(\uvdif\adif)\gdif^{-1} (\uvdif\bdif^{-1}), & \delta = -1.
\end{cases}
\end{equation*}

{\rm 2)}~The partition of $\EDiffWU$ into the orbits of this action does not depend on a particular choice of $\uvdif\in\DiffUVW$, and will be denoted by
\[
    \dbli{\EDiffWU}{\DiffW}
\]
and called \term{$(\EDiffWU,\pm)$-double cosets}.

{\rm 3)}~The correspondence
\begin{equation*}
\Aatlas=\{(\Udomain,\Ucmap),(\Vdomain,\Vcmap)\}
\ \longmapsto \
\Vcmap^{-1}\circ\Ucmap
\end{equation*}
associating to each $\UVAtlas$-atlas $\Aatlas$ on $\Aspace$ its transition map $\Atrmap\in\DiffW$ yields an \term{injection} $\mu$
\begin{itemize}
\item
of the set $\CkAUV\,/\,\HomeoLPartUV$ of all $\Ck$-structures on $\Aspace$ induced by some $\UVAtlas$-atlas up to a $\Ck$-diffeomorphism leaving invariant $\Udomain$ and $\Vdomain$ or exchanging them
\item
into the set \ $\dbli{\EDiffWU}{\DiffW}$ \ of $(\EDiffWU,\pm)$-double cosets.
\end{itemize}
Moreover, $\mu$ is a bijection if and only if every $\gdif\in\DiffW$ is a transition map for some $\UVAtlas$-atlas $\Aatlas$.
\end{subtheorem}
\begin{proof}
1) Let $\gdif\in\EDiffWU$.
Note that the unit of $\EDiffWU\wr\bZ_{2}$ is $e = (\id_{\Wdomain}, \id_{\Wdomain}, +1)$ and $e\cdot \gdif = \gdif$.
Therefore, it remains to check associativity axiom of the action.
Let $\adif,\bdif,\cdif,\ddif\in\EDiffWU$.
Then
\begin{align*}
(\cdif,\ddif,1)\cdot\bigl( (\adif,\bdif,1)\cdot\gdif \bigr)
&= (\cdif,\ddif,1)\cdot (\uvdif\bdif\uvdif^{-1}\gdif\adif^{-1})
 = (\uvdif\ddif\uvdif^{-1}) (\uvdif\bdif\uvdif^{-1}\gdif\adif^{-1})\cdif^{-1} \\
&= \uvdif(\ddif\bdif)\uvdif^{-1}\gdif(\cdif\adif)^{-1} =
   (\cdif\adif, \ddif\bdif, 1)\cdot \gdif =
   \bigl[ (\cdif,\ddif,1)(\adif,\bdif,1)\bigr]\cdot\gdif
\\[1mm]
(\cdif,\ddif,-1)\cdot\bigl( (\adif,\bdif,1)\cdot\gdif \bigr)
&= (\cdif,\ddif,-1)\cdot(\uvdif\bdif\uvdif^{-1}\gdif\adif^{-1})
 = (\uvdif\cdif) (\uvdif\bdif\uvdif^{-1}\gdif\adif^{-1})^{-1} (\uvdif\ddif^{-1}) \\
&= \uvdif\cdif \adif  \gdif^{-1}\uvdif\bdif^{-1}\uvdif^{-1} \uvdif\ddif^{-1}
 = \uvdif (\cdif \adif) \, \gdif^{-1} \, \uvdif (\ddif\bdif)^{-1} \\
&= (\cdif\adif,\ddif\bdif,-1)\cdot \gdif
 = \bigl[ (\cdif,\ddif,-1)(\adif,\bdif,1)\bigr]\cdot\gdif,
\\[1mm]
(\cdif,\ddif,1)\cdot\bigl( (\adif,\bdif,-1)\cdot\gdif \bigr)
&= (\cdif,\ddif,1)\cdot (\uvdif\adif\, \gdif^{-1}\, \uvdif\bdif^{-1}) =
   (\uvdif\ddif\uvdif^{-1})(\uvdif\adif\, \gdif^{-1}\, \uvdif\bdif^{-1})\cdif^{-1}  \\
&= (\uvdif\ddif\adif) \gdif^{-1} \uvdif(\cdif\bdif)^{-1}
 = (\ddif\adif,\cdif\bdif,-1)\cdot \gdif  \\
&= \bigl[ (\cdif,\ddif,1)(\adif,\bdif,-1)\bigr]\cdot\gdif,
\\[1mm]
(\cdif,\ddif,-1)\cdot\bigl( (\adif,\bdif,-1)\cdot\gdif \bigr)
&= (\cdif,\ddif,-1)\cdot (\uvdif\adif\, \gdif^{-1}\, \uvdif\bdif^{-1}) =
   (\uvdif\cdif)(\uvdif\adif\, \gdif^{-1}\, \uvdif\bdif^{-1})^{-1}(\uvdif\ddif^{-1})  \\
&= \uvdif\cdif\bdif\uvdif^{-1}\,  \gdif \, \adif^{-1}\uvdif^{-1}\uvdif\ddif^{-1}
 = (\uvdif\cdif\bdif\uvdif^{-1})  \gdif (\ddif\adif)^{-1} \\
&= (\ddif\adif,\cdif\bdif,1)\cdot \gdif
 = \bigl[ (\cdif,\ddif,-1)(\adif,\bdif,-1)\bigr]\cdot\gdif.
\end{align*}

\newcommand\buvdif{\bar{\uvdif}}
\newcommand\bact{*}
2) Let $\buvdif\colon\mpair{\Udomain}{\Uatlas}\to\mpair{\Vdomain}{\Vatlas}$ be some other $\Cr{\kk}$-diffeomorphism such that $\buvdif(\Wdomain) = \Wdomain$.
For $(\adif,\bdif,\delta)\in\EDiffWU\wr\bZ_{2}$ and $\gdif\in\DiffW$ denote by $(\adif,\bdif,1)\cdot\gdif$ and $(\adif,\bdif,1) \bact \gdif$ the actions of $(\adif,\bdif,\delta)$ on $\gdif$ constructed via $\uvdif$ and $\buvdif$.
Then $\gamma:=\buvdif^{-1}\uvdif$ belongs to $\DiffUW$, whence
for any $\adif,\bdif\in\EDiffWU$ and $\gdif\in\DiffW$ we have that
\begin{align*}
(\adif,\bdif,1)\cdot \gdif &=
(\uvdif\bdif\uvdif^{-1}) \ \gdif \ \adif^{-1} =
\buvdif(\buvdif^{-1}\uvdif)\bdif(\uvdif^{-1}\buvdif)\buvdif^{-1} \ \gdif \ \adif^{-1} = \\ &=
\buvdif(\gamma\bdif\gamma^{-1}) \ \gdif \ \adif^{-1} =
(\adif,\gamma\bdif\gamma^{-1},1) * \gdif,
\\
(\adif,\bdif,-1)\cdot \gdif &=
(\uvdif\adif) \ \gdif^{-1} \ \uvdif\bdif^{-1} =
\buvdif(\buvdif^{-1}\uvdif)\adif \ \gdif^{-1} \ \buvdif(\buvdif^{-1}\uvdif)\bdif^{-1} = \\ &=
\buvdif(\gamma\adif) \ \gdif \ \buvdif(\gamma\bdif^{-1}) =
(\gamma\adif,\bdif\gamma^{-1},-1) * \gdif.
\end{align*}
This shows that the orbits of $\gdif$ with respect to the actions defined via $\uvdif$ and $\buvdif$ coincide.

3) For an $\UVAtlas$-atlas $\Aatlas$ denote by $[\Aatlas]$ its equivalence class in
\[ \CkAUV\,/\,\HomeoLPartUV,\]
and let $[\Atrmap]$ be the $(\EDiffWU,\pm)$-double coset in $\dbli{\EDiffWU}{\DiffW}$ of its transition map $\Atrmap$.

Let $\Aatlas$ and $\Batlas$ be two $\UVAtlas$-atlases.
Then by Theorem~\ref{th:CkLUV_HLUV} and Lemma~\ref{lm:h_UV_VU} the following conditions are equivalent:
\begin{enumerate}[label={\rm(\alph*)}]
\item\label{enum:exch:impl_h_LA_LB}
there exists a $\Ck$-diffeomorphism $\dif\colon\mpair{\Aspace}{\Aatlas}\to\mpair{\Aspace}{\Batlas}$
\item\label{enum:exch:gb1__b_ga_ainv}
$\erestr{\bdif}{\Wdomain} \circ \Atrmap \circ \erestr{\adif^{-1}}{\Wdomain}=\Btrmap$ for some $\adif\in\DiffUW$ and $\bdif\in\DiffVW$,
or
$\erestr{\bdif}{\Wdomain} \circ \Atrmap \circ \erestr{\adif^{-1}}{\Wdomain}=\Btrmap$ for some $\adif\in\DiffUVW$ and $\bdif\in\DiffVUW$.
\end{enumerate}
Notice that the relations~\ref{enum:exch:gb1__b_ga_ainv} just mean commutativity of the respective squares in the Diagram~\eqref{equ:motivation_for_actions}, i.e.\ the identities under those diagrams.
In turn, those identities are equivalent to another condition:
\begin{enumerate}[label={\rm(\alph*)}, resume]
\item\label{enum:exch:orb}
the transition maps $\Atrmap$ and $\Btrmap$ belong to the same orbit in $\dbli{\EDiffWU}{\DiffW}$.
\end{enumerate}

Now again, the implication~\ref{enum:exch:impl_h_LA_LB}$\Rightarrow$\ref{enum:exch:orb}, means that the correspondence
\[
\CkAUV\,/\,\HomeoLUV \ \ni \ [\Aatlas] \ \xrightarrow{~~\mu~~} \ [\Atrmap] \ \in \ \dbli{\EDiffWU}{\DiffW}
\]
is well-defined, while the inverse implication~\ref{enum:exch:orb}$\Rightarrow$\ref{enum:exch:impl_h_LA_LB} shows that $\mu$ is injective.
\end{proof}

\subsection{Application. Differentiable structures on the line $\DLine$ with two origins}
\label{sect:diff_struct_on_L}
Let
\[
    \DLine = (\bR\times\{0,1\})/\{ (x,0)\sim(x,1) \ \text{for} \ x\in\Rzp\}.
\]
be the topological space obtained by gluing two copies of $\bR$ via the identity homeomorphism of $\Rzp$, see Figure~\ref{fig:LY}.
Let also $\pi\colon\RTwoCopies \to \DLine$ be the corresponding quotient map.
Denote
\begin{gather*}
    \Udomain = \pi(\bR\times\{0\}),\qquad
    \Vdomain = \pi(\bR\times\{1\}),\\
    \Wdomain = \Udomain \cap \Vdomain
    = \pi\bigl((\Rzp) \times\{0\}\bigr)
    = \pi\bigl((\Rzp) \times\{1\}\bigr).
\end{gather*}
As in the case of $\Ylet$, it is easy to see that every homeomorphism $\dif\colon\DLine\to\DLine$ either leaves $\Udomain$ and $\Vdomain$ invariant or exchanges them.
In particular, $\dif$ always preserves $\Wdomain$.
\begin{figure}[htbp!]
\includegraphics[height=2.5cm]{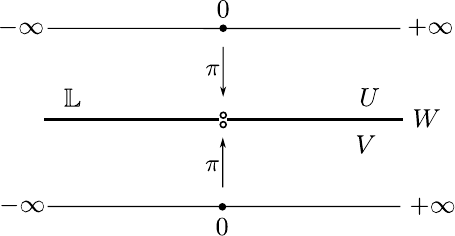}
\caption{Manifold $\DLine$}\label{fig:l_line}
\end{figure}
Moreover, the restriction of $\pi$ onto $\bR\times\{0\}$ and $\bR\times\{1\}$ are open embeddings,
\begin{align*}
    \canUcmap &= (\restr{\pi}{\bR\times\{0\}})^{-1}\colon\Udomain\to\bR\times\{0\} \equiv \bR, &
    \canVcmap &= (\restr{\pi}{\bR\times\{1\}})^{-1}\colon\Vdomain\to\bR\times\{1\} \equiv \bR,
\end{align*}
whence the inverse maps can be regarded as charts $(\Udomain,\canUcmap)$ and $(\Vdomain,\canVcmap)$ of $\Ylet$.
Note that the corresponding transition map $\canVcmap\circ\canUcmap^{-1} = \id_{\Rzp}$, whence the atlas $\CanonAtlas{\Ylet}=\{ (\Udomain,\canUcmap), (\Vdomain,\canVcmap) \}$ on $\DLine$ is $\Ck$ for all $\kk=1,\ldots,\infty$.

Let $\mathcal{D}:=\XDIFF{+,\kk}{}{\bR\setminus0}$ be the group of orientation-preserving $\Ck$-diffeomorphisms of $\bR\setminus0$, and $\mathcal{W}$ be its subgroup consisting of diffeomorphisms which can be extended to a $\Ck$-diffeomorphism of $\bR$.
Since $\canUcmap$ and $\canVcmap$ are diffeomorphisms, they yield isomorphisms
\begin{align*}
    &\canUcmap_{*}\colon\DiffUW \to \mathcal{D}, && \canUcmap_{*}(\dif) = \canUcmap\circ\dif\circ\canUcmap_{*}^{-1}, \\
    &\canVcmap_{*}\colon\DiffVW \to \mathcal{D}, && \canVcmap_{*}(\dif) = \canVcmap\circ\dif\circ\canVcmap_{*}^{-1}.
\end{align*}
It is also evident, that $\canUcmap_{*}(\EDiffWU) = \canVcmap_{*}(\EDiffWV) = \mathcal{W}$, whence $\canUcmap_{*}$ and $\canVcmap_{*}$ yield bijections
\begin{gather*}
\dbl{\EDiffWV}{\DiffW}{\EDiffWU}  \equiv \dbl{\mathcal{W}}{\mathcal{D}}{\mathcal{W}},\\
\dbli{\EDiffWV}{\DiffW}  \equiv \dbli{\mathcal{W}}{\mathcal{D}}.
\end{gather*}
Now Theorems~\ref{th:CkLUV_HLUV} and~\ref{th:act_DUWZ2_DW} imply show that $\dbl{\mathcal{W}}{\mathcal{D}}{\mathcal{W}}$ classify all $\Ck$-structures on $\DLine$, while $\dbli{\mathcal{W}}{\mathcal{D}}$ classifies such structures up to a $\Ck$-diffeomorphism leaving $\Udomain$ and $\Vdomain$ invariant.
This includes statement of Theorem~\ref{th:Ck_struct_on_L_summary}.

\begin{subremark}\rm
The assumption on diffeomorphisms $\dif\colon\Aspace\to\Aspace$ to preserve both $\Udomain$ and $\Vdomain$ looks rather restrictive, since usually there might exist a lot of other diffeomorphisms of $\Aspace$ which do not preserve $\Udomain$ and $\Vdomain$.
On the other hand, as we see for the line with two origins $\DLine$ and non-Hausdorff letter $\Ylet$, they do admit such cover $\Udomain\cup\Vdomain$ being invariant under homeomorphisms.
\end{subremark}

\begin{subremark}\rm
Suppose $\Aspace$ is a union of $n$ open subsets $\Udomain[i]$, $i=1,\ldots,n$, each endowed with a $\Ck$-atlas $\Uatlas_{i}$, so that those atlases are $\Ck$-compatible on the respective intersections.
Then \term{theoretically} one can apply Theorems~\ref{th:CkLUV_HLUV} and~\ref{th:act_DUWZ2_DW} to classify \term{by induction} such $\Ck$-structures on $\Aspace$ up to a $\Ck$-diffeomorphism leaving each $\Udomain[i]$ invariant.
Namely, at first one can describe such structures on $\Vdomain[1] = \Udomain[1] \cup \Udomain[2]$, then on $\Vdomain[2] := \Vdomain[1] \cup \Udomain[3]$, and so on.

Such a scheme appears in the classification of differentiable structures given by Munkres, see~\cite[Definition~6.9]{Munkres:AnnMath:1968} of \term{pasting a ball} into an oriented manifold.
Moreover, \cite[Theorem~6.12]{Munkres:AnnMath:1968} claims that differentiable structures on $\Aspace$ obtained by pasting a ball are in one-to one correspondence with a certain quotient $Q^{n}$ of the group $\Gamma^n$ of isotopy classes of diffeomorphisms of $S^{n-1}$ up to diffeomorphisms extendable to diffeomorphisms of $D^{n}$.

This statement is close to Theorem~\ref{th:CkLUV_HLUV}, but we will discuss the precise relationships in another paper.
\end{subremark}

\subsection{Application. Uniqueness of differentiable structures on $\bR$}
The following statement gives a sufficient conditions for the transitivity of the action of $\HomeoLUV$ on $\CkAUV$.
Its particular case is used in the proof of uniqueness of $\Ck$-structures on the real line, see Remark~\ref{rem:joinable_charts_on_R} below.

\begin{sublemma}\label{lm:diff_to_canon_struct}
Let $\Latlas=\{ (\Udomain, \id_{\Udomain}), (\Vdomain, \id_{\Vdomain})\}$ be the initial $\Ck$-atlas of $\Aspace$, see~\eqref{equ:atlas_C__U_V}.
Let also $\Aatlas=\{(\Udomain,\Ucmap),(\Vdomain,\Vcmap)\}$ be a $\UVAtlas$-atlas, and $\Atrmap = \Vcmap\circ\Ucmap^{-1}$ be its transition map.
Suppose that there exist an open neighborhood $\hUdomain$ of $\Aman := \Udomain\setminus\Vdomain$ and an open neighborhood $\hVdomain$ of $\Bman := \Vdomain\setminus\Udomain$, and a $\Ck$-diffeomorphism $\bdif\colon\Vdomain\to\Vdomain$ such that
\begin{enumerate}[label={\rm(\alph*)}]
\item\label{enum:lm:diff_to_canon_struct:aW_W} $\bdif(\Wdomain)=\Wdomain$;
\item\label{enum:lm:diff_to_canon_struct:a_fix_on_hU} $\bdif$ is fixed on $\hVdomain$;
\item\label{enum:lm:diff_to_canon_struct:a_WhV__h_WhV} $\restr{\bdif}{\Wdomain\cap\hVdomain} = \restr{\Atrmap}{\Wdomain\cap\hVdomain}$.
\end{enumerate}
Then there exists a $\Ck$-diffeomorphism $\adif\colon\Udomain\to\Udomain$ such that $\adif(\Wdomain)=\Wdomain$ and
\begin{equation}\label{equ:gA_b_idW_ainv}
    \Atrmap = \erestr{\bdif}{\Wdomain}\circ\erestr{\adif^{-1}}{\Wdomain}
            = \erestr{\bdif}{\Wdomain}\circ\id_{\Wdomain}\circ\erestr{\adif^{-1}}{\Wdomain}.
\end{equation}
In particular,
$\Atrmap$ and $\id_{\Wdomain}$ belong to the same $(\EDiffWV,\EDiffWU)$-double cosets in $\dbl{\EDiffWV}{\DiffW}{\EDiffWU}$, and therefore $\mpair{\Aspace}{\Latlas}$ and $\mpair{\Aspace}{\Aatlas}$ are $\Ck$-diffeomorphic.
In particular, the map
\[
\dif\colon\Aspace\to\Aspace,
\qquad
\dif = \begin{cases}
\adif\circ\Ucmap, & \text{on} \ \Udomain, \\
\bdif\circ\Vcmap, & \text{on} \ \Vdomain
\end{cases}
\]
is a $\Ck$-diffeomorphism $\mpair{\Aspace}{\Aatlas}\to\mpair{\Aspace}{\Latlas}$ which coincides with $\Ucmap$ on $\hUdomain$ and with $\Vcmap$ on $\hVdomain$.
\end{sublemma}
\begin{figure}[htbp!]
\includegraphics[height=2.5cm]{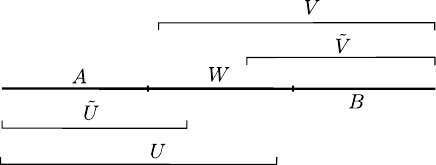}
\caption{}\label{fig:uniq_str}
\end{figure}
\begin{proof}
Put $\hadif = \Atrmap^{-1}\circ\erestr{\bdif}{\Wdomain}$.
Then by~\ref{enum:lm:diff_to_canon_struct:a_WhV__h_WhV}, $\hadif$ is fixed on $\Wdomain\cap\hUdomain$, whence it extends by the identity to a $\Ck$-diffeomorphism $\adif$ of all of $\Udomain$ such that $\adif(\Wdomain)=\Wdomain$.
It then follows that $\Atrmap = \hbdif\circ\erestr{\adif^{-1}}{\Wdomain} = \erestr{\bdif}{\Wdomain}\circ\id_{\Wdomain}\circ\erestr{\adif^{-1}}{\Wdomain}$, so~\eqref{equ:gA_b_idW_ainv} holds.
The corresponding properties of $\dif$ are straightforward and agree with~\eqref{equ:h_Y_Y__reconstr_via_a_b}.
\end{proof}

\begin{subremark}\label{rem:joinable_charts_on_R}\rm
Let $\Aspace = (0;3)$ be the open interval in $\bR$ with the canonical $\Ck$-structure $\Latlas = \CanonAtlas{\Aspace}$, $\Udomain = (0;2)$, $\Vdomain = (1;3)$, and $\Wdomain = \Udomain\cap\Vdomain=(1;2)$.
Suppose two homeomorphisms onto $\Ucmap\colon\Udomain\to(0;2)$ and $\Vcmap\colon\Vdomain\to(1;3)$ constitute a $\UVAtlas$-atlas $\Aatlas$ such that $\Ucmap(\Wdomain)=\Vcmap(\Wdomain)=(1;2) = \Wdomain$.
Then by the standard gluing technique, see e.g.\ \cite[Lemma~4.2.1]{LysynskyiMaksymenko:SmoothStr:2024}, they satisfy assumptions of Lemma~\ref{lm:diff_to_canon_struct}.
Hence there is a $\Ck$-diffeomorphism $\dif\colon\mpair{\Aspace}{\Aatlas} \to \mpair{\Aspace}{\Latlas}$.
Moreover it coincides with $\Ucmap$ on $(0;1+\eps)$ and with $\Vcmap$ on $(2-\eps;3)$ for some small $\eps>0$.

Note that this diffeomorphism can be regarded as a single chart on $\Aspace$ being $\Ck$-compatible with $\mpair{\Aspace}{\Aatlas}$, and thus one can replace those charts with this one.
This observation is a key technical step in the proof of Theorem~\ref{th:uniq_ck_struct_on_R}, see e.g.\ \cite[Lemma~4.2.1]{LysynskyiMaksymenko:SmoothStr:2024}.
\end{subremark}

\section{Categorical point of view}
\label{sect:categorical_view}
In this section we extend previous results to general categories.
In what follows $\Ccateg$ is a subcategory of some category $\Hcateg$.
For simplicity we will restrict ourselves only to considering isomorphisms and automorphisms.

One can think of $\Hcateg$ to be the category of topological spaces and their homeomorphisms, while $\Ccateg$ is a subcategory of $\Ck$-manifolds and their $\Ck$-diffeomorphisms.

\subsection{Isomorphisms between objects and morphisms}
For a pair of objects $\Udomain$ and $\Vdomain$ in $\Hcateg$ denote by $\MorH{\Udomain}{\Vdomain}$ the subcategory of morphisms $\Udomain\to\Vdomain$ between them in $\Hcateg$, and by $\IsomH{\Udomain}{\Vdomain}$ the subcategory of \term{isomorphisms} (invertible morphisms) $\Udomain\to\Vdomain$.

If $\Udomain=\Vdomain$, then $\IsomH{\Udomain}{\Udomain}$ is also denoted by $\AutHU$, and called the \term{group of automorphisms of $\Udomain$}.
It has a single object $\Udomain$ whose morphisms consist of all invertible morphisms $\Udomain\to\Udomain$.

We also have a \term{morphism category of $\Hcateg$} denoted by $\ArrHcateg$ whose \term{objects} are morphisms from $\Hcateg$, and \term{morphisms} between pairs of morphisms $\wuinc\colon\Wdomain\to\Udomain$ and $\wuinc'\colon\Wdomain'\to\Udomain'$ are commutative diagrams of the form:
\[
    \xymatrix{
        \Wman \ar[r]^-{\wuinc} \ar[d]_-{\wdif}  & \Uman \ar[d]^-{\udif} \\
        \Wman' \ar[r]^-{\wuinc'}                & \Uman'
    }
\]
If in this diagram $\wdif$ and $\udif$ are isomorphisms, then $(\wdif,\udif)$ is an \term{isomorphism} between $\wuinc$ and $\wuinc'$.

We will be interested in the subcategories $\IsomCat{\ArrHcateg}{\wuinc}{\wuinc'}$ and $\AutCObject[\ArrHcateg]{\wuinc}$ of isomorphisms between morphisms $\wuinc$ and $\wuinc'$, and self-isomorphisms of $\wuinc$ respectively.

\subsection{$\Cspan$-Spans and their $\HCatlas$-atlases}
By a \term{$\Cspan$-span} we will mean the following diagram in the category $\Ccateg$:
\begin{equation}\label{equ:c_span}
    \Aspace\colon\spanCat{\Udomain}{\wuinc}{\Wdomain}{\wvinc}{\Vdomain},
\end{equation}
i.e.\ the objects $\Wdomain,\Udomain,\Vdomain$ and the arrows belong to $\Ccateg$.

If we want to distinguish the order of arrows, i.e.\ regard $\wuinc$ as the first arrow of $\Aspace$, while $\wvinc$ is the second one, then $\Aspace$ will be called an \term{ordered $\Cspan$-span}.
In this case we define the \term{reversed $\Cspan$-span} to $\Aspace$ by
\begin{equation}\label{equ:c_span_rev}
    \spanRev{\Aspace} \colon \spanCat{\Vdomain}{\wvinc}{\Wdomain}{\wuinc}{\Udomain}.
\end{equation}
An \term{$\Hcateg$-morphism} between $\Ccateg$-spans
\begin{align*}
    &\Aspace\colon\spanCat{\Vdomain}{\wvinc}{\Wdomain}{\wuinc}{\Udomain}, &
    &\Bspace\colon\spanCat{\hVdomain}{\hwvinc}{\hWdomain}{\hwuinc}{\hUdomain}
\end{align*}
is the following commutative diagram:
\begin{equation}\label{equ:span_morphism}
    \xymatrix{
        \Udomain \ar[d]_-{\udif} &
        \Wdomain  \ar[l]_-{\wuinc}  \ar[r]^-{\wvinc}  \ar[d]^-{\wdif} &
        \Vdomain \ar[d]^-{\vdif}  \\
        \hUdomain             &
        \hWdomain \ar[l]_-{\hwuinc} \ar[r]^-{\hwvinc}     &
        \hVdomain
    }
\end{equation}
in which $\udif$, $\wdif$, and $\vdif$ are \term{morphisms in $\Hcateg$}.
Note that $\Cspan$-spans and their morphisms constitute a category.
\begin{subdefinition}\label{def:HC_atlas}
An \term{$\HCatlas$-atlas} is a pair $(\Aspace,\Aatlas)$, where $\Aspace \colon \spanCat{\Udomain}{\wuinc}{\Wdomain}{\wvinc}{\Vdomain}$ is a $\Cspan$-span, and
\[
    \Aatlas = \{
        (\Ucmap,\hUcmap)\in\AutCObject[\ArrHcateg]{\wuinc}, \
        (\Vcmap,\hVcmap)\in\AutCObject[\ArrHcateg]{\wvinc}
    \}
\]
is a pair of automorphisms (called \term{charts}) of arrows of $\Aspace$ such that the following isomorphism (called the \term{transition map})
\[
    \Atrmap:=\hVcmap\circ\hUcmap^{-1}\colon\Wdomain\to\Wdomain
\]
is an \term{isomorphism in $\Ccateg$}, i.e.\ belongs to $\AutCW$.

We will also denote $(\Aspace, \Aatlas)$ by $\mpair{\Aspace}{\Aatlas}$ and say that \term{$\Aatlas$ is an $\HCatlas$-atlas on $\Aspace$}.
\end{subdefinition}
Thus, an $\HCatlas$-atlas is a commutative diagram in $\Hcateg$
\[
\xymatrix@C=5em@R=1.5em{
    \Udomain
    &
    \Wdomain
        \ar@{_(->}[l]_-{\wuinc}
        \ar[r]^-{\Atrmap \,=\, \hVcmap \,\circ\, \hUcmap^{-1}}
    &
    \Wdomain
        \ar@{^(->}[r]^-{\wvinc}
    &
    \Vdomain
    \\
    \Udomain
        \ar[u]_-{\Ucmap}
    &
    \Wdomain
        \ar@{_(->}[l]_-{\wuinc}
        \ar@{=}[r]^-{\id_{\Wdomain}}
        \ar[u]^-{\hUcmap}
    &
    \Wdomain
        \ar@{^(->}[r]^-{\wvinc}
        \ar[u]_-{\hVcmap}
    &
    \Vdomain
        \ar[u]^-{\Vcmap}
    }
\]
in which vertical arrows are isomorphisms from $\Hcateg$ while horizontal arrows are in $\Ccateg$.

Let
\begin{align*}
\mpair{\Aspace}{\Aatlas}&=
\bigl(
    \Aspace\colon\spanCat{\Udomain}{\wuinc}{\Wdomain}{\wvinc}{\Vdomain}, \
    \Aatlas=\{ (\aUcmap,\haUcmap),(\aVcmap,\haVcmap)\}
\bigr),
\\
\mpair{\Bspace}{\Batlas}&=
\bigl(
    \Bspace\colon\spanCat{\hUdomain}{\hwuinc}{\hWdomain}{\hwuinc}{\hVdomain}, \
    \Batlas=\{(\bUcmap,\hbUcmap),(\bVcmap,\hbVcmap)\}
\bigr)
\end{align*}
be two $\HCatlas$-atlases and $\dif=(\udif,\wdif,\vdif)\colon\Aspace\to\Bspace$ be a morphism between the corresponding $\Cspan$-spans.
Then in the following commutative diagram similar to~\eqref{equ:h_UU_VV__Y} and~\eqref{equ:h_UU_VV__L_general}:
\begin{equation}\label{equ:h_UU_VV__L_categories}
\begin{gathered}
\xymatrix@C=5em@R=1.5em{
    \Udomain
        \ar@{-->}@/_2.5ex/[ddd]_-{\adif}
    &
    \Wdomain
        \ar@{_(->}[l]_-{\wuinc}
        \ar[r]^-{\Atrmap = \haVcmap \circ \haUcmap^{-1}}
        \ar@{-->}@/^2.5ex/[ddd]^-{\hadif}
    &
    \Wdomain
        \ar@{^(->}[r]^-{\wvinc}
        \ar@{-->}@/_2.5ex/[ddd]_-{\hbdif}
    &
    \Vdomain
        \ar@{-->}@/^2.5ex/[ddd]^-{\bdif}
    \\
    \Udomain
        \ar[u]_-{\aUcmap}
        \ar[d]^-{\udif}
    &
    \Wdomain
        \ar@{_(->}[l]_-{\wuinc}
        \ar@{=}[r]
        \ar[u]^-{\haUcmap}
        \ar[d]_-{\wdif}
    &
    \Wdomain
        \ar@{^(->}[r]^-{\wvinc}
        \ar[u]_-{\haVcmap}
        \ar[d]^-{\wdif}
    &
    \Vdomain
        \ar[u]^-{\aVcmap}
        \ar[d]_-{\vdif}
    \\
    \hUdomain
        \ar[d]^-{\bUcmap}
    &
    \hWdomain
        \ar@{_(->}[l]_-{\hwuinc}
        \ar@{=}[r]
        \ar[d]_-{\hbUcmap}
    &
    \hWdomain
        \ar@{^(->}[r]^-{\hwvinc}
        \ar[d]^-{\hbVcmap}
    &
    \hVdomain
        \ar[d]_-{\bVcmap}
    \\
    \hUdomain
    &
    \hWdomain
        \ar@{_(->}[l]_-{\hwuinc}
        \ar[r]^-{\Btrmap = \hbVcmap\circ\hbUcmap^{-1}}
    &
    \hWdomain
        \ar@{^(->}[r]^-{\hwvinc}
    &
    \hVdomain
}
\end{gathered}
\end{equation}
the dashed arrows
\begin{align*}
    & \bigl(
        \adif=\bUcmap\circ\udif\circ\aUcmap^{-1}, \
        \hadif=\hbUcmap\circ\wdif\circ\haUcmap^{-1}
    \bigr)
    \in
    \MorCat{\ArrHcateg}{\wuinc}{\hwuinc}, \\
    & \bigl(
        \bdif=\bVcmap\circ\vdif\circ\aVcmap^{-1}, \
        \hbdif=\hbVcmap\circ\wdif\circ\haVcmap^{-1}
    \bigr)
    \in
    \MorCat{\ArrHcateg}{\wvinc}{\hwvinc}
\end{align*}
are called \term{representation of $\dif$} (with respect to the atlases $\Aatlas$ and $\Batlas$).

We will say that $\dif$ is a \term{morphism between $\HCatlas$-atlases} $\mpair{\Aspace}{\Aatlas}\to\mpair{\Bspace}{\Batlas}$ if its representations are in fact morphisms in $\ArrCcateg$, i.e.\ $(\adif, \hadif)\in \MorCat{\ArrCcateg}{\wuinc}{\hwuinc}$ and $(\bdif, \hbdif)\in \MorCat{\ArrCcateg}{\wvinc}{\hwvinc}$.

The following lemma is easy and follows from Diagram~\eqref{equ:h_UU_VV__L_categories}.
\begin{sublemma}
Let $\bigl(
    \Aspace\colon\spanCat{\Udomain}{\wuinc}{\Wdomain}{\wvinc}{\Vdomain}, \
    \Aatlas=\{(\aUcmap,\haUcmap),(\aVcmap,\haVcmap)\}
\bigr)$ be an $\HCatlas$-atlas and $\dif=(\udif,\wdif,\vdif)$, see~\eqref{equ:span_morphism}, be an \term{isomorphism} to some other $\Cspan$-span $\Aspace\colon\spanCat{\hUdomain}{\hwuinc}{\hWdomain}{\hwuinc}{\hVdomain}$.
Put
\begin{align*}
\bUcmap  &= \adif\circ\aUcmap\circ\udif^{-1}   \in \AutH{\Udomain}, &
\hbUcmap &= \hadif\circ\haUcmap\circ\wdif^{-1} \in \AutH{\Wdomain}, \\
\bVcmap  &= \bdif\circ\aVcmap\circ\vdif^{-1}   \in \AutH{\hUdomain},&
\hbVcmap &= \hbdif\circ\hbUcmap\circ\wdif^{-1} \in \AutH{\hWdomain}.
\end{align*}
Then $\indatl{\dif}{\Aatlas}=\{(\bUcmap,\hbUcmap),(\bVcmap,\hbVcmap)\}$ is an $\HCatlas$-atlas on $\Bspace$ and $\dif$ is an isomorphism of atlases $\mpair{\Aspace}{\Aatlas}\to\mpair{\Bspace}{\indatl{\dif}{\Aatlas}}$.
Moreover, the transition map of $\indatl{\dif}{\Aatlas}$ is $\hbVcmap\circ\hbUcmap^{-1} = \hbdif\circ(\haVcmap\circ\haUcmap^{-1})\circ\hadif^{-1}$.
\end{sublemma}

\newcommand\pUW{p_{\Udomain}}
\newcommand\pVW{p_{\Vdomain}}

\subsection{Isomorphism classes of $\HCatlas$-atlases}
Let $\Aspace\colon\spanCat{\Udomain}{\wuinc}{\Wdomain}{\wvinc}{\Vdomain}$ be a $\Cspan$-span.
Recall that, by definition, each of the categories $\AutCObject[\ArrCcateg]{\wuinc}$, $\AutCObject[\ArrCcateg]{\wvinc}$, and $\AutCW$ consists of a single object.
Then we have the following two functors:
\begin{align*}
    \pUW&\colon\AutCObject[\ArrCcateg]{\wuinc} \to \AutCW, & \pUW(\adif,\hadif)&=\hadif, \\
    \pVW&\colon\AutCObject[\ArrCcateg]{\wvinc} \to \AutCW, & \pUW(\bdif,\hbdif)&=\hbdif,
\end{align*}
being the projections onto the second coordinates.
They can also be regarded as homomorphisms of the respective groups.

Denote by
\[
   \AutWExtU := \pUW(\AutCObject[\ArrCcateg]{\wuinc}),
   \qquad
   \AutWExtV := \pVW(\AutCObject[\ArrCcateg]{\wuinc})
\]
their ``\term{images}''.
More precisely, $\AutWExtU$ is a subcategory of $\AutCW$ with a single object $\Wdomain$, and whose morphisms are automorphisms $\hadif\in\AutCW$ for which there exists an automorphism $\adif\in\AutCU$ such that $(\adif,\hadif)\in\AutCObject[\ArrCcateg]{\wuinc}$, that is $\wuinc\circ\adif=\hadif\circ\wuinc$.
A similar description holds for $\AutWExtV$.

Then the product $\AutWExtU\times\AutWExtV$ acts on morphisms of $\AutCW$ (i.e.\ of $\Wdomain$) by the following rule:
\[
    (\hadif,\hbdif)\cdot\gdif := \hbdif\circ\gdif\circ\hadif^{-1}
\]
for $\hadif\in\AutWExtU$, $\hbdif\in\AutWExtV$, and $\gdif\in\AutCW$.
The corresponding orbits set will be denoted by
\[
    \dbl{\AutWExtV}{\AutCW}{\AutWExtU}
\]
and called \term{$(\AutWExtU,\AutWExtV)$-double cosets}.

The proof of the following theorem almost literally repeats arguments of Theorem~\ref{th:CkLUV_HLUV} and we leave it for the reader.

\begin{subtheorem}\label{th:CkLUV_HLUV:categories}
Let $\Ccateg$ be a subcategory of a category $\Hcateg$.
Then for every ordered $\Ccateg$-span
\[ \Aspace\colon\spanCat{\Udomain}{\wuinc}{\Wdomain}{\wvinc}{\Vdomain}\]
the correspondence
\begin{equation*}
\Aatlas=\{(\aUcmap,\haUcmap), (\aVcmap,\haVcmap)\}
\ \longmapsto \
\Atrmap = \haVcmap\circ\haUcmap^{-1}
\end{equation*}
associating to each $\HCatlas$-atlas $\Aatlas$ on $\Aspace$ its transition map $\Atrmap$ yields an \term{injective} map $\mu'$
\begin{itemize}
\item
from the isomorphism classes of $\HCatlas$-atlases on $\Aspace$
\item
into the double cosets $\dbl{\AutWExtV}{\AutCW}{\AutWExtU}$.
\end{itemize}
Moreover, $\mu'$ is a \term{bijection} if and only if every $\gdif\in\AutCW$ is a transition map for some $\HCatlas$-atlas $\Aatlas$.
\end{subtheorem}

\subsection{$\pm$-Isomorphism classes of $\HCatlas$-atlases}
As above, let $\Aspace\colon\spanCat{\Udomain}{\wuinc}{\Wdomain}{\wvinc}{\Vdomain}$ be an ordered $\Cspan$-span.
Suppose also that $\Udomain$ and $\Vdomain$ are isomorphic in $\Ccateg$.
Then one may expect that there is an isomorphism between $\Aspace$ and its reversed $\Cspan$-span $\spanRev{\Aspace}\colon\spanCat{\Vdomain}{\wvinc}{\Wdomain}{\wuinc}{\Udomain}$, \eqref{equ:c_span_rev}.

Suppose also that we have two $\HCatlas$-atlases
\begin{align*}
    \Aatlas &= \{(\aUcmap,\haUcmap),(\aVcmap,\haVcmap)\}, &
    \Batlas &= \{(\bUcmap,\hbUcmap),(\bVcmap,\hbVcmap)\}
\end{align*}
on $\Aspace$, and $\dif=(\udif,\wdif,\vdif)\colon\mpair{\Aspace}{\Aatlas}\to\mpair{(\spanRev{\Aspace})}{\Batlas}$ be an isomorphism, so ``\term{$\dif$ exchanges $\Udomain$ and $\Vdomain$}''.
Then we have the following commutative diagram similar to~\eqref{equ:h_UV_VU__Y} and~\eqref{equ:h_UV_VU__L_general}:
\begin{equation}\label{equ:h_UV_VU__L_categories}
\begin{gathered}
\xymatrix@C=5em@R=1.5em{
    \Udomain
        \ar@{-->}@/_2.5ex/[ddd]_-{\adif}
    &
    \Wdomain
        \ar@{_(->}[l]_-{\wuinc}
        \ar[r]^-{\Atrmap = \haVcmap \circ \haUcmap^{-1}}
        \ar@{-->}@/^2.5ex/[ddd]^-{\hadif}
    &
    \Wdomain
        \ar@{^(->}[r]^-{\wvinc}
        \ar@{-->}@/_2.5ex/[ddd]_-{\hbdif}
    &
    \Vdomain
        \ar@{-->}@/^2.5ex/[ddd]^-{\bdif}
    \\
    \Udomain
        \ar[u]_-{\aUcmap}
        \ar[d]^-{\udif}
    &
    \Wdomain
        \ar@{_(->}[l]_-{\wuinc}
        \ar@{=}[r]
        \ar[u]^-{\haUcmap}
        \ar[d]_-{\wdif}
    &
    \Wdomain
        \ar@{^(->}[r]^-{\wvinc}
        \ar[u]_-{\haVcmap}
        \ar[d]^-{\wdif}
    &
    \Vdomain
        \ar[u]^-{\aVcmap}
        \ar[d]_-{\vdif}
    \\
    \Vdomain
        \ar[d]^-{\bVcmap}
    &
    \Wdomain
        \ar@{_(->}[l]_-{\wvinc}
        \ar@{=}[r]
        \ar[d]_-{\bVcmap}
    &
    \Wdomain
        \ar@{^(->}[r]^-{\wuinc}
        \ar[d]^-{\bUcmap}
    &
    \Udomain
        \ar[d]_-{\bUcmap}
    \\
    \Vdomain
    &
    \Wdomain
        \ar@{_(->}[l]_-{\wvinc}
    &
    \Wdomain
        \ar@{^(->}[r]^-{\wuinc}
        \ar[l]_-{\Btrmap = \bVcmap\circ\bUcmap^{-1}}
    &
    \Udomain
}
\end{gathered}
\end{equation}

Isomorphisms between $\mpair{\Aspace}{\Aatlas}\to\mpair{\Aspace}{\Batlas}$ and $\mpair{\Aspace}{\Aatlas}\to\mpair{(\spanRev{\Aspace})}{\Batlas}$ will be called \term{$\pm$-isomorphisms}.

Our aim is to describe the $\pm$-isomorphism classes of atlases on $\Aspace$.

Regard the group $\bZ_{2}=\{\pm1\}$ as a category with a single object and two morphisms.
Denote by $\AutWExtU\wr\bZ_{2}$ the \term{wreath product} of $\AutWExtU$ and $\bZ_{2}$.
By definition, this is a still category with one object and in which the multiplication of morphisms is given by
\begin{align*}
    &(\cdif,\ddif,\delta)(\adif,\bdif,+1)=(\cdif\adif,\ddif\bdif,\delta),&
    &(\cdif,\ddif,\delta)(\adif,\bdif,-1)=(\ddif\adif,\cdif\bdif,\delta),
\end{align*}
for morphisms $\adif,\bdif,\cdif,\ddif\in\AutWExtU$ and $\delta=\pm1$.

The following theorem can be proved similarly to Theorem~\ref{th:act_DUWZ2_DW}.
\begin{subtheorem}\label{lm:act_DUWZ2_DW_categories}
Let $\Ccateg$ be a subcategory of a category $\Hcateg$ and
$\Aspace\colon\spanCat{\Udomain}{\wuinc}{\Wdomain}{\wvinc}{\Vdomain}$
be a $\Cspan$-span.

{\rm 1)}~For every $\eta\in\IsomC{\Udomain}{\Vdomain}$ there is an action of $\AutWExtU\wr\bZ_{2}$ on (morphisms of) $\AutCW$ given by the following rule: if $(\adif,\bdif,\delta)\in\AutWExtU\wr\bZ_{2}$ and $\gdif\in\AutCW$, then
\begin{equation*}
(\adif,\bdif,\delta)\cdot\gdif =
\begin{cases}
\uvdif (\bdif\uvdif^{-1} \gdif \adif^{-1}) =
(\uvdif\bdif\uvdif^{-1}) \gdif \adif^{-1},  & \delta = +1,
\\
\uvdif (\bdif\uvdif^{-1} \gdif \adif^{-1})^{-1}  =
(\uvdif\adif)\gdif^{-1} (\uvdif\bdif^{-1}), & \delta = -1.
\end{cases}
\end{equation*}

{\rm 2)}~The partition of $\AutWExtU$ into the orbits of this action does not depend on a particular choice of $\uvdif\in\IsomC{\Udomain}{\Vdomain}$, and will be denoted by
\[
    \dbli{\AutWExtU}{\AutCW}
\]
and called \term{ $(\AutWExtU,\pm)$-double cosets}.

{\rm 3)}~The correspondence
\begin{equation*}
\Aatlas=\{(\Udomain,\Ucmap),(\Vdomain,\Vcmap)\}
\ \longmapsto \
\Atrmap:=\Vcmap^{-1}\circ\Ucmap
\end{equation*}
associating to each $\HCatlas$-atlas $\Aatlas$ on $\Aspace$ its transition map $\Atrmap\in\AutCW$ yields an \term{injection} $\mu$
\begin{itemize}
\item
from the $\pm$-isomorphism classes of $\HCatlas$-atlases on $\Aspace$
\item
into the $\pm$-double cosets $\dbli{\AutWExtU}{\AutCW}$.
\end{itemize}
Moreover, $\mu$ is a bijection if and only if every $\gdif\in\AutCW$ is a transition map for some $\HCatlas$-atlas $\Aatlas$.
\end{subtheorem}

Concluding, let us briefly mention two new aspects of the above theorems when we pass to general categories.

1) In contrast to the case of diffeomorphisms, in categories we have no ``restriction map``, that is if $(\adif,\hadif)\in\AutCObject[\ArrCcateg]{\wuinc}$, then $\adif$ do not uniquely determine $\hadif$.

2) The above theorems are also applicable for different pairs of categories.
For instance, if $1\leq l<k \leq \infty$, then one can ask about classification of $\Cr{k}$-structures (or complex structures) on some manifold $\Aspace$ up to a $\Cr{l}$-diffeomorphism.
Moreover, $\Aspace$ can be an infinite-dimensional manifold.

\subsection{Application. Characterization of double cosets}
\label{sect:charact_double_cosets}

Let $\Hcateg$ be a group, and $\Ccateg$ be its subgroup.
Regard them as categories with one object $\Wdomain$ and all arrows being isomorphisms.
As a simple illustration of Theorem~\ref{th:CkLUV_HLUV:categories} we will present a characterization of double cosets $\dbl{\Ccateg}{\Hcateg}{\Ccateg}$.

Consider the following span $\Aspace\colon\spanCat{\Wdomain}{\id_{\Wdomain}}{\Wdomain}{\id_{\Wdomain}}{\Wdomain}$.
Then a pair of elements $\Aatlas = (\Ucmap,\Vcmap) \in \Hcateg\times\Hcateg$ is a \term{$\HCatlas$-atlas} (for $\Aspace$), if the ``\term{transition map}'' $\Atrmap:=\Vcmap\circ\Ucmap^{-1} \in \Ccateg$.
In other words, $\Ucmap,\Vcmap$ belongs to the same right adjacent class $\Hcateg/\Ccateg$.

Further, let $\Aatlas = (\aUcmap,\aVcmap)$ and $\Batlas = (\bUcmap,\bVcmap)$ be two $\HCatlas$-atlases, an $\dif\in\Hcateg$.
Then $\dif$ is
\begin{itemize}
\item an \term{isomorphism} $\mpair{\Aspace}{\Aatlas} \to \mpair{\Aspace}{\Batlas}$ if its \term{representations} $\adif:=\bUcmap\circ\dif\circ\aUcmap^{-1}$ and $\bdif:=\bVcmap\circ\dif\circ\aVcmap^{-1}$ belong to $\Ccateg$, and

\item an \term{isomorphism} $\mpair{\Aspace}{\Aatlas}\to\mpair{(\spanRev{\Aspace})}{\Batlas}$ if \term{representations} $\adif:=\bVcmap\circ\dif\circ\aUcmap^{-1}$ and $\bdif:=\bUcmap\circ\dif\circ\aVcmap^{-1}$ belong to $\Ccateg$.
\end{itemize}
Then Theorems~\ref{th:CkLUV_HLUV:categories} and~\ref{lm:act_DUWZ2_DW_categories} imply the following:
\begin{subcorollary}\label{cor:char_double_cosets}
Two $\HCatlas$-atlases $\Aatlas,\Batlas\in\Hcateg$ are isomorphic (resp.\ $\pm$-isomorphic) iff their transition maps $\Atrmap$ and $\Btrmap$ belong to the same double coset $\dbl{\Ccateg}{\Hcateg}{\Ccateg}$ (resp.\ $\pm$-double coset $\dbli{\Ccateg}{\Hcateg}$).

In particular, the correspondence
\[
    \Hcateg\times\Hcateg \ni (\Ucmap,\Vcmap) = \Aatlas \longmapsto \Atrmap = \Vcmap\circ\Ucmap^{-1} \in \Ccateg
\]
yields a bijection between
\begin{itemize}
\item isomorphism classes of $\HCatlas$-atlases and double cosets $\dbl{\Ccateg}{\Hcateg}{\Ccateg}$;
\end{itemize}
and also between
\begin{itemize}
\item $\pm$-isomorphism classes of $\HCatlas$-atlases and $\pm$-double cosets $\dbli{\Ccateg}{\Hcateg}$.
\end{itemize}
\end{subcorollary}

\subsection*{Acknowledgement}
This work was supported by a grant from the Simons Foundation (SFI-PD-Ukraine-00014586, L.M.V. and S.I.M.).
The authors are grateful to Taras Bakakh for discussions of foundations of sets theory, and to Volodymyr Lyubashenko and Bohdan Feshchenko for discussions of categories.

\def\cprime{$'$} \def\cprime{$'$} \def\cprime{$'$} \def\cprime{$'$}
  \def\cprime{$'$} \def\cprime{$'$} \def\cprime{$'$} \def\cprime{$'$}
  \def\cprime{$'$} \def\cprime{$'$} \def\cprime{$'$} \def\cprime{$'$}
  \def\cprime{$'$} \def\cprime{$'$}


\begin{thebibliography}{10}

\bibitem{Akbulut:PM:2012}
Selman Akbulut.
\newblock Exotic structures on smooth four-manifolds.
\newblock In {\em Perspectives in analysis, geometry, and topology}, volume 296
  of {\em Progr. Math.}, pages 1--17. Birkh\"auser/Springer, New York, 2012.
\newblock \href {https://doi.org/10.1007/978-0-8176-8277-4\_1}
  {\path{doi:10.1007/978-0-8176-8277-4\_1}}.

\bibitem{Cairns:AnnM:1936}
Stewart~S. Cairns.
\newblock Polyhedral approximations to regular loci.
\newblock {\em Ann. of Math. (2)}, 37(2):409--415, 1936.
\newblock \href {https://doi.org/10.2307/1968452} {\path{doi:10.2307/1968452}}.

\bibitem{ChernovNemirovski:CMPh:2013}
Vladimir Chernov and Stefan Nemirovski.
\newblock Cosmic censorship of smooth structures.
\newblock {\em Comm. Math. Phys.}, 320(2):469--473, 2013.
\newblock \href {https://doi.org/10.1007/s00220-013-1686-1}
  {\path{doi:10.1007/s00220-013-1686-1}}.

\bibitem{Connell:AnnM:1963}
E.~H. Connell.
\newblock Approximating stable homeomorphisms by piecewise linear ones.
\newblock {\em Ann. of Math. (2)}, 78:326--338, 1963.
\newblock \href {https://doi.org/10.2307/1970346} {\path{doi:10.2307/1970346}}.

\bibitem{Donaldson:JDG:1983}
S.~K. Donaldson.
\newblock An application of gauge theory to four-dimensional topology.
\newblock {\em J. Differential Geom.}, 18(2):279--315, 1983.
\newblock URL: \url{http://projecteuclid.org/euclid.jdg/1214437665}.

\bibitem{FreedUhlenbeck:Inst:1991}
Daniel~S. Freed and Karen~K. Uhlenbeck.
\newblock {\em Instantons and four-manifolds}, volume~1 of {\em Mathematical
  Sciences Research Institute Publications}.
\newblock Springer-Verlag, New York, second edition, 1991.
\newblock \href {https://doi.org/10.1007/978-1-4613-9703-8}
  {\path{doi:10.1007/978-1-4613-9703-8}}.

\bibitem{Freedman:JDG:1982}
Michael~Hartley Freedman.
\newblock The topology of four-dimensional manifolds.
\newblock {\em J. Differential Geometry}, 17(3):357--453, 1982.
\newblock URL: \url{http://projecteuclid.org/euclid.jdg/1214437136}.

\bibitem{Gauld:NonMetrManif:2014}
David Gauld.
\newblock {\em Non-metrisable manifolds}.
\newblock Springer, Singapore, 2014.
\newblock \href {https://doi.org/10.1007/978-981-287-257-9}
  {\path{doi:10.1007/978-981-287-257-9}}.

\bibitem{Gompf:JDG:1985}
Robert~E. Gompf.
\newblock An infinite set of exotic {${\bf R}^4$}'s.
\newblock {\em J. Differential Geom.}, 21(2):283--300, 1985.
\newblock URL: \url{http://projecteuclid.org/euclid.jdg/1214439566}.

\bibitem{HaefligerReeb:EM:1957}
Andr{\'e} Haefliger and Georges Reeb.
\newblock Vari\'et\'es (non s\'epar\'ees) \`a une dimension et structures
  feuillet\'ees du plan.
\newblock {\em Enseignement Math. (2)}, 3:107--125, 1957.

\bibitem{KirbySiebenmann:Smooth:1977}
Robion~C. Kirby and Laurence~C. Siebenmann.
\newblock {\em Foundational essays on topological manifolds, smoothings, and
  triangulations}.
\newblock Annals of Mathematics Studies, No. 88. Princeton University Press,
  Princeton, NJ; University of Tokyo Press, Tokyo, 1977.
\newblock With notes by John Milnor and Michael Atiyah.

\bibitem{LysynskyiMaksymenko:SmoothStr:2024}
Mykola Lysynskyi and Sergiy Maksymenko.
\newblock Classification of differentiable structures on the non-hausdorff line
  with two origins, 2024.
\newblock \href {https://doi.org/10.48550/arXiv.2406.09576}
  {\path{doi:10.48550/arXiv.2406.09576}}.

\bibitem{Moise_V:AnnM:1952}
Edwin~E. Moise.
\newblock Affine structures in {$3$}-manifolds. {V}. {T}he triangulation
  theorem and {H}auptvermutung.
\newblock {\em Ann. of Math. (2)}, 56:96--114, 1952.
\newblock \href {https://doi.org/10.2307/1969769} {\path{doi:10.2307/1969769}}.

\bibitem{Munkres:AnnMath:1968}
James Munkres.
\newblock Obstructions to the smoothing of piecewise-differentiable
  homeomorphisms.
\newblock {\em Ann. of Math. (2)}, 72:521--554, 1960.
\newblock \href {https://doi.org/10.2307/1970228} {\path{doi:10.2307/1970228}}.

\bibitem{Munkres:PhD:1955}
James~R. Munkres.
\newblock {\em Some applications of triangulation theorems}.
\newblock PhD thesis, 1955.
\newblock Thesis (Ph.D.)--University of Michigan.
\newblock URL: \url{https://www.proquest.com/docview/301954388}.

\bibitem{Nyikos:AM:1992}
Peter~J. Nyikos.
\newblock Various smoothings of the long line and their tangent bundles.
\newblock {\em Adv. Math.}, 93(2):129--213, 1992.
\newblock \href {https://doi.org/10.1016/0001-8708(92)90027-I}
  {\path{doi:10.1016/0001-8708(92)90027-I}}.

\bibitem{Stallings:PCPS:1962}
John Stallings.
\newblock The piecewise-linear structure of {E}uclidean space.
\newblock {\em Proc. Cambridge Philos. Soc.}, 58:481--488, 1962.

\bibitem{Taubes:JDG:1987}
Clifford~Henry Taubes.
\newblock Gauge theory on asymptotically periodic {$4$}-manifolds.
\newblock {\em J. Differential Geom.}, 25(3):363--430, 1987.
\newblock URL: \url{http://projecteuclid.org/euclid.jdg/1214440981}.

\bibitem{Whitehead:AnnM:1940}
J.~H.~C. Whitehead.
\newblock On {$C^1$}-complexes.
\newblock {\em Ann. of Math. (2)}, 41:809--824, 1940.
\newblock \href {https://doi.org/10.2307/1968861} {\path{doi:10.2307/1968861}}.

\bibitem{Whitehead:AnnMath:1961}
J.~H.~C. Whitehead.
\newblock Manifolds with transverse fields in euclidean space.
\newblock {\em Ann. of Math. (2)}, 73:154--212, 1961.
\newblock \href {https://doi.org/10.2307/1970286} {\path{doi:10.2307/1970286}}.

\end{thebibliography}

\end{document}